\documentclass[reqno,11pt]{article}
\usepackage{graphicx, amsmath, amsthm,amssymb}  
\usepackage{tikz}
\usetikzlibrary{arrows}
\usepgflibrary{fpu}

\input{stochbif.sty}
\def\hper{h^{\text{\rm{per}}}}
\def\Law{\mathsf{Law}}

\begin{document}


\title{Noise-induced phase slips, log-periodic oscillations,\\ 
and the Gumbel distribution}
\author{Nils Berglund} 
\date{}   

\maketitle

\begin{abstract}
When two synchronised phase oscillators are perturbed by weak noise, they
display occasional losses of synchrony, called phase slips. The slips can be
characterised by their location in phase space and their duration. We show that
when properly normalised, their location converges, in the vanishing noise
limit, to the sum of an asymptotically geometric random variable and a Gumbel
random variable. The duration also converges to a Gumbel variable with
different parameters. We relate these results to recent works on the phenomenon
of log-periodic oscillations and on links between transition path theory and
extreme-value theory. 
\end{abstract}

\leftline{\small{\it Date.\/} 
March 28, 2014. Revised. October 2, 2014. 
}
\leftline{\small 2000 {\it Mathematical Subject Classification.\/} 
60H10, 		
34F05  		
(primary), 
60G70,   	
34D06   		
(secondary).
}


\section{Introduction}
\label{sec_intro}

The aim of this work is to explore connections between the concepts given in the
title:
\begin{enum}
\item 	\emph{Noise-induced phase slips} are occasional losses of synchrony of
two coupled phase oscillators, due to stochastic perturbations~\cite{PRK}. The
problem of finding the distribution of their location and length can be
formulated as a stochastic exit problem, which involves the exit through a
so-called \emph{characteristic boundary}~\cite{Day7,Day3}. 

\item 	\emph{Log-periodic oscillations} designate the periodic dependence of a
quantity of interest, such as a power law exponent, on the logarithm of a system
parameter. They often occur in systems presenting a discrete scale
invariance~\cite{Sornette_98}. In the context of the stochastic exit problem,
they are connected to the phenomenon of \emph{cycling} of the exit distribution
through an unstable periodic orbit~\cite{Day6,BG7,BG_periodic2}. 

\item 	The \emph{Gumbel distribution} is one of the max-stable distributions
known from extreme-value theory~\cite{Gnedenko_1943}. This distribution has
been known to occur in the exit distribution through an unstable periodic
orbit~\cite{MS4,BG7,BG_periodic2}. More recently, the Gumbel distribution has
also been found to govern the length of \emph{reactive paths} in
one-dimensional exit
problems~\cite{CerouGuyaderLelievreMalrieu12,Bakhtin_2013a,Bakhtin_2014a}. 
\end{enum}

In this work, we review a number of prior results on exit distributions, and
build on them to derive properties of the phase slip distributions. We start in
Section~\ref{sec_synchro} by recalling the classical situation of two coupled
phase oscillators, and the phenomenology of noise-induced phase slips. In
Section~\ref{sec_exit}, we present the mathematical set-up for the systems we
will consider, and introduce different tools that are used for the study of the
stochastic exit problem.

Section~\ref{sec_logper} contains our main results on the distribution of the
phase slip location. These results are mainly based on those
of~\cite{BG_periodic2} on the exit distribution through an unstable planar
periodic orbit, slightly reformulated in the context of limit distributions. We
also discuss links to the concept of log-periodic oscillations. 

In Section~\ref{sec_Gumbel}, we discuss a number of connections to extreme-value
theory. After summarizing properties of the Gumbel distribution relevant to our
problem, we give a short review of recent results by C\'erou, Guyader,
Leli\`evre and Malrieu~\cite{CerouGuyaderLelievreMalrieu12} and by
Bakhtin~\cite{Bakhtin_2013a,Bakhtin_2014a} on the appearance of the Gumbel
distribution in transition path theory for one-dimensional problems.

Section~\ref{sec_slips} presents our results on the duration of phase slips,
which build on the previous results from transition path theory. 
Section~\ref{sec_conclusion} contains a summary and some open questions, while 
the proofs of the main theorems are contained in the Appendix. 

\subsection*{Acknowledgments}

This work is based on a presentation given at the meeting \lq\lq Inhomogeneous
Random Systems\rq\rq\ at Institut Henri Poincar\'e, Paris, on January 28, 2014.
It is a pleasure to thank Giambattista Giacomin for inviting me, and Fran\c cois
Dunlop, Thierry Gobron and Ellen Saada for organising this interesting meeting.
I am grateful to Barbara Gentz for critical comments on the manuscript, and to
Arkady Pikovsky for suggesting to look at the phase slip duration. The
connection with elliptic functions was pointed out to me by G\'erard Letac.
Finally, thanks are due to an anonymous referee for comments leading to an
improved presentation. 


\section{Synchronization of phase oscillators}
\label{sec_synchro}

In this section, we briefly recall the setting of two coupled phase
oscillators showing synchronization, following mainly~\cite{PRK}. 


\subsection{Deterministic phase locking}
\label{ssec_synchrodet}

Consider two oscillators, whose dynamics is governed by ordinary differential
equations (ODEs) of the form 
\begin{align}
 \nonumber
 \dot x_1 &= f_1(x_1)\;, \\
 \dot x_2 &= f_2(x_2)\;, 
 \label{sdet01}
\end{align} 
where $x_1\in\R\!^{n_1}, x_2\in\R\!^{n_2}$ with $n_1,n_2\geqs 2$. 
A classical example of a system displaying oscillations is the Van der Pol
oscillator~\cite{vanderPol20,vanderPol26,vanderPol27}
\begin{equation}
 \label{sdet02}
 \ddot \theta_i - \gamma_i (1 - \theta_i^2) \dot\theta_i + \theta_i = 0\;,
\end{equation} 
which can be transformed into a first-order system by setting
$x_i=(\theta_i,\dot\theta_i)\in\R^2$. The precise form of the vector fields
$f_i$, however, does not matter. What is important is that each system admits an
asymptotically stable periodic orbit, called a \emph{limit cycle}. These limit
cycles can be parametrised by angular variables $\phi_1,\phi_2\in\R/\Z$ in such
a way that 
\begin{align}
 \nonumber
 \dot \phi_1 &= \omega_1\;,\\
 \dot \phi_2 &= \omega_2\;,
 \label{sdet03}
\end{align} 
where $\omega_1,\omega_2\in\R$ are constant angular
frequencies~\cite[Section~7.1]{PRK}. Note that the product of the two limit
cycles forms a two-dimensional invariant torus in phase space $\R\!^{n_1+n_2}$. 

Consider now a perturbation of System~\eqref{sdet01} in which the oscillators
interact weakly, given by 
\begin{align}
 \nonumber
 \dot x_1 &= f_1(x_1) + \eps g_1(x_1,x_2)\;, \\
 \dot x_2 &= f_2(x_2) + \eps g_2(x_1,x_2)\;.  
 \label{sdet04}
\end{align} 
The theory of normally hyperbolic invariant manifolds (see for
instance~\cite{HirschPughShub}) shows that the invariant torus persists for
sufficiently small nonzero $\eps$ (for stronger coupling, new phenomena such as
oscillation death can occur~\cite[Section~8.2.2]{PRK}). For small $\eps$, the
reduced equations~\eqref{sdet03} for the dynamics on the torus take the form 
\begin{align}
 \nonumber
 \dot \phi_1 &= \omega_1 + \eps Q_1(\phi_1,\phi_2)\;,\\
 \dot \phi_2 &= \omega_2 + \eps Q_2(\phi_1,\phi_2)\;,
 \label{sdet05}
\end{align} 
where $Q_{1,2}$ can be computed perturbatively in terms of $f_{1,2}$ and
$g_{1,2}$. 
Assume that the natural frequencies $\omega_1, \omega_2$ are different, but that
the \emph{detuning}  $\nu=\omega_2-\omega_1$ is small. Introducing new variables
$\psi=\phi_1-\phi_2$ and $\ph = (\phi_1+\phi_2)/2$ yields a system of the form 
\begin{align}
 \nonumber
 \dot \psi &= -\nu + \eps q(\psi,\ph)\;,\\
 \dot \ph &= \omega + \Order{\eps}\;, 
 \label{sdet06}
\end{align} 
where $\omega=(\omega_1+\omega_2)/2$ is the mean frequency. Note that the phase
difference $\psi$ evolves more slowly than the mean phase $\ph$, so that the
theory of averaging applies~\cite{Bogoliubov56,Verhulst}. For small $\nu$ and
$\eps$, solutions of~\eqref{sdet06} are close to those of the averaged system 
\begin{equation}
 \label{sdet07}
 \omega \frac{\6\psi}{\6\ph} =  -\nu + \eps \bar q(\psi)\;, 
 \qquad \bar q(\psi) = \int_0^1 q(\psi,\ph) \6\ph
\end{equation} 
(recall our convention that the period is equal to $1$). 
In particular, solutions of the equation $-\nu + \eps \bar q(\psi)=0$
correspond to stationary solutions of the averaged equation~\eqref{sdet07}, and
to periodic orbits of the original equation~\eqref{sdet06} (and
thus also of~\eqref{sdet05}). 

For example, in the case of \emph{Adler's equation}, $\bar q(\psi) =
\sin(2\pi\psi)$, there are two stationary points whenever
$\abs{\nu}<\abs{\eps}$. They give rise to one stable and one unstable periodic
orbit. The stable periodic orbit corresponds to a synchronized state, because
the phase difference $\psi$ remains bounded for all times. This is the
phenomenon known as \emph{phase locking}. 

\begin{remark}
Similar phase locking phenomena appear when the ratio $\omega_2/\omega_1$ is
close to any rational number $m/n\in\Q$. Then for small $\eps$ the quantity
$n\phi_1-m\phi_2$ may stay bounded for all times ($n\colon m$ \emph{frequency
locking}). The sets of parameter values $(\eps,\nu)$ for which frequency locking
with a specific ratio occurs are known as \emph{Arnold tongues}
\cite{Arnold1961}. 
\end{remark}


\subsection{Noise-induced phase slips}
\label{ssec_slips}

\begin{figure}
\begin{center}
\begin{tikzpicture}[>=stealth',main
node/.style={draw,thick,circle,blue,fill=blue!20,minimum
size=5pt,inner sep=0pt},scale=0.55,x=1.25cm,y=0.5cm, 
declare function={
pot(\x) = cos(4*\x r) - 1.5*\x;
}
]



\newcommand*{\xmin}{0}
\newcommand*{\xmax}{8}
\newcommand*{\ymin}{-13.5}

\path[fill=teal!20,thick,-,smooth,domain=\xmin:\xmax,samples=60,/pgf/fpu,
/pgf/fpu/output format=fixed] 
plot (\x, {pot(\x) }) -- (\xmax,{pot(\xmax)}) -- (\xmax,\ymin) -- (\xmin,\ymin)
--
(\xmin,{pot(\xmin)});

\draw[teal,thick,-,smooth,domain=\xmin:\xmax,samples=60,/pgf/fpu,
/pgf/fpu/output
format=fixed] plot (\x, {pot(\x) });


\path[fill=white] plot (\xmin,\ymin) -- (\xmax,\ymin) -- (\xmax,\ymin-1.5) --
(\xmin,\ymin-1.5);


\node[main node] at (2.45,{pot(2.45)+0.3}) {}; 


\draw[blue,semithick,->] (2.7,-3.7) .. controls (3,-3) and (3.2,-3) ..
(3.5,-3.7);

\node[] at (0.3,2.7) {{\bf (a)}};

\end{tikzpicture}
\hspace{5mm}
\begin{tikzpicture}[>=stealth',main
node/.style={draw,semithick,circle,fill=white,minimum
size=2pt,inner sep=0pt},scale=0.55,x=2cm,y=0.8cm, 
declare function={
stab(\x) = 1.5-0.6*sin(\x r+4);
unstab(\x) = 5+0.8*cos(\x r);
trans(\x) = stab(\x) + 3.14*(1+tanh(2*(\x-3.14)));
}
]



\path[fill=green!30,smooth,domain=0:6.28,samples=60,/pgf/fpu,
/pgf/fpu/output format=fixed] plot (\x, {stab(\x)+0.5}) -- plot(6.28-\x,
{stab(6.28-\x)-0.5});
\path[fill=green!30,smooth,domain=0:6.28,samples=60,/pgf/fpu,
/pgf/fpu/output format=fixed] plot (\x, {stab(\x)+6.28+0.5}) -- plot(6.28-\x,
{stab(6.28-\x)+6.28-0.5});


\draw[->,thick] (0,0) -- (0,9.3);
\draw[->,thick] (0,0) -- (6.8,0);


\draw[green!50!black,thick,-,smooth,domain=0:6.28,samples=60,/pgf/fpu,
/pgf/fpu/output format=fixed] plot (\x, {stab(\x)});
\draw[green!50!black,thick,-,smooth,domain=0:6.28,samples=60,/pgf/fpu,
/pgf/fpu/output format=fixed] plot (\x, {6.28+stab(\x)});


\draw[violet,thick,-,smooth,domain=0:6.28,samples=60,/pgf/fpu,
/pgf/fpu/output format=fixed] plot (\x, {unstab(\x)});


\pgfmathsetseed{16825527}
\draw[red,semithick,-,smooth,domain=0:6.28,samples=60,/pgf/fpu,
/pgf/fpu/output format=fixed] plot (\x, {trans(\x) + 0.5*rand});


\newcommand*{\xa}{2.2}
\draw[dashed] ({\xa},0) -- ({\xa},{stab(\xa)+0.5});
\node[main node] at ({\xa},0) {};
\node[main node] at ({\xa},{stab(\xa)+0.5}) {};
\node[] at ({\xa+0.1},-0.6) {$\ph_{\tau_-}$};

\renewcommand*{\xa}{3.05}
\draw[dashed] ({\xa},0) -- ({\xa},{unstab(\xa)});
\node[main node] at ({\xa},0) {};
\node[main node] at ({\xa},{unstab(\xa)}) {};
\node[] at ({\xa+0.1},-0.6) {$\ph_{\tau_0}$};

\renewcommand*{\xa}{3.7}
\draw[dashed] ({\xa},0) -- ({\xa},{stab(\xa)+6.28-0.5});
\node[main node] at ({\xa},0) {};
\node[main node] at ({\xa},{stab(\xa)+6.28-0.5}) {};
\node[] at ({\xa+0.1},-0.6) {$\ph_{\tau_+}$};


\node[] at (6.3,-0.6) {$\ph$};
\node[] at (-0.3,8.5) {$\psi$};

\node[] at (-0.2,10) {{\bf (b)}};

\end{tikzpicture}
\vspace{-6mm}
\end{center}
\caption[]{{\bf (a)} Washboard potential $V(\psi)$ of the averaged
system~\eqref{slips01}, in a case where $\nu<0$. Phase slips correspond to
transitions over a local potential maximum. {\bf (b)} For the unaveraged
system~\eqref{slips03}, phase slips involve crossing the unstable periodic orbit
delimiting the basin of attraction of the synchronized state. 
}
\label{fig_phase_slip}
\end{figure}
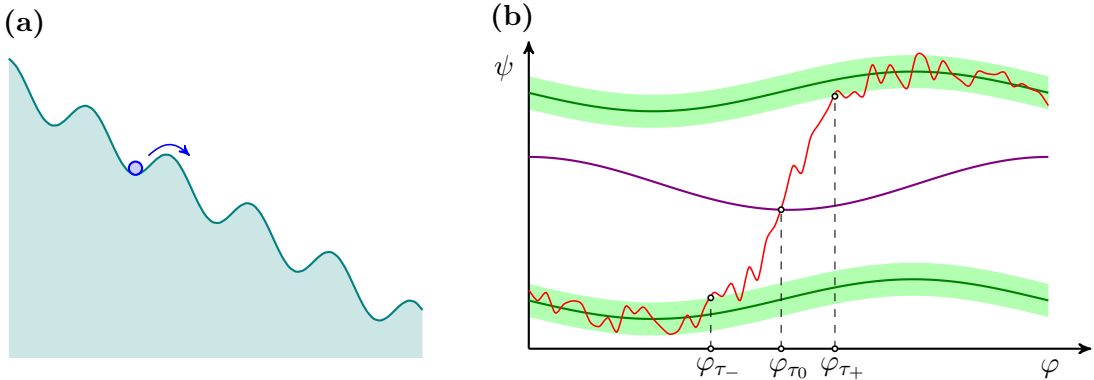

Consider now what happens when noise is added to the system. This is often 
done (see e.g.~\cite[Chapter~9]{PRK}) by looking at the effect of noise on
the averaged system~\eqref{sdet07}, which becomes 
\begin{equation}
 \label{slips01}
 \omega \frac{\6\psi}{\6\ph} =  -\nu + \eps \bar q(\psi) 
 + \text{noise}\;,
\end{equation} 
where we will specify in the next section what kind of noise we consider. The
first two terms on the right-hand side of~\eqref{slips01} can be written as 
\begin{equation}
 \label{slips02}
 -\frac{\partial}{\partial\psi} V(\psi)\;, 
 \qquad \text{where } 
 V(\psi) = \nu\psi - \eps\int_0^\psi \bar q(x)\6x\;.
\end{equation} 
In the synchronization region, the potential $V(\psi)$ has the shape of a tilted
periodic, or \emph{washboard} potential (\figref{fig_phase_slip}a). The local
minima of the potential represent the synchronized state, while the local maxima
represent an unstable state delimiting the basin of attraction of the
synchronized state. In the absence of noise, trajectories are attracted
exponentially fast to the synchronized state and stay there. When weak noise is
added, solutions still spend most of the time in a neighbourhood of a
synchronized state. However, occasional transitions through the unstable state
may occur, meaning that the system temporarily desynchronizes, before returning
to synchrony. This behaviour is called a \defwd{phase slip}. Transitions in both
directions may occur, that is, $\psi$ can increase or decrease by $1$ per phase
slip. When detuning and noise are small, however, transitions over the lower
local maximum of the washboard potential are more likely.  

In reality, however, we should add noise to the unaveraged 
system~\eqref{sdet06}, which becomes  
\begin{align}
 \nonumber
 \dot \psi &= -\nu + \eps q(\psi,\ph)+ \text{noise}\;,\\
 \dot \ph &= \omega + \Order{\eps}+ \text{noise}\;.  
 \label{slips03}
\end{align} 
Phase slips are now associated with transitions across the unstable orbit
(\figref{fig_phase_slip}b). Two important random quantities characterising the
phase slips are 
\begin{enum}
\item 	the value of the phase $\ph_{\tau_0}$ at the time $\tau_0$ when the
unstable orbit is crossed, and
\item 	the duration of the phase slip, which can be defined as the phase
difference between the time $\tau_-$ when a suitably defined neighbourhood of
the stable orbit is left, and the time $\tau_+$ when a neighbourhood of (a
translated copy of) the stable orbit is reached. 
\end{enum}
Unless the system~\eqref{slips03} is independent of
the phase $\ph_{\tau_0}$, there is no reason for the slip phases to have a
uniform
distribution. Our aim is to determine the weak-noise asymptotics of the phase
$\ph_{\tau_0}$ and of the phase slip duration $\ph_{\tau_+}-\ph_{\tau_-}$. 


\section{The stochastic exit problem}
\label{sec_exit}

Let us now specify the mathematical set-up of our analysis. The stochastically
perturbed systems that we consider are It\^o stochastic differential equations
(SDEs) of the form 
\begin{equation}
 \label{exit01}
 \6x_t = f(x_t)\6t + \sigma g(x_t)\6W_t\;,
\end{equation} 
where $x_t$ takes values in $\R^2$, and $W_t$ denotes $k$-dimensional standard
Brownian motion, for some $k\geqs2$. Physically, this describes the situation
of Gaussian white noise with a state-dependent amplitude $g(x)$. Of course, one
may consider more general types of noise, such as time-correlated noise, but
such a setting is beyond the scope of the present analysis.

The drift term $f$ and the diffusion term $g$ are assumed to satisfy the usual
regularity assumptions guaranteeing the existence of a unique strong solution
for all square-integrable initial conditions $x_0$ (see for
instance~\cite[Section~5.2]{Oksendal}). In addition, we assume that $g$
satisfies the uniform ellipticity condition 
\begin{equation}
 \label{exit02}
 c_1 \norm{\xi}^2 \leqs 
 \langle \xi, D(x) \xi \rangle
 \leqs c_2 \norm{\xi}^2
 \qquad \forall x,\xi\in\R^2\;,
\end{equation} 
where $c_2\geqs c_1>0$. Here $D(x)=gg^\text{T}(x)$ denotes the \emph{diffusion
matrix}. 

We finally assume that the drift term $f(x)$ results from a system of the
form~\eqref{sdet06}, in a synchronized case where there is one stable and one
unstable orbit. It will be convenient to choose coordinates $x=(\ph,r)$ such
that the unstable periodic orbit is given by $r=0$ and the stable orbit is
given by $r=1/2$. The original system is defined on a torus, but we will unfold
everything to the plane $\R^2$, considering $f$ (and $g$) to be 
periodic with period $1$ in both variables $\ph$ and $r$. The resulting system
has the form\footnote{In the situation of synchronised phase oscillators
considered here, the change of variables yielding~\eqref{exit03} is global. In
more general situations considered in~\cite{BG_periodic2}, such a
transformation may only exist locally, in a domain surrounding the stable and
unstable orbit, and the analysis applies up to the first-exit time from that
domain.}
\begin{align}
 \nonumber
 \6r_t &= f_r(r_t,\ph_t)\6t + \sigma g_r(r_t,\ph_t)\6W_t\;, \\
 \6\ph_t &= f_\ph(r_t,\ph_t)\6t + \sigma g_\ph(r_t,\ph_t)\6W_t\;, 
 \label{exit03}
\end{align}
and admits unstable orbits of the form $\set{r=n}$ and stable orbits of the
form $\set{r=n+1/2}$ for any integer $n$. In particular, $f_r(n/2,\ph)=0$ for
all $n\in\Z$ and all $\ph\in\R$. Using a so-called
equal-time parametrisation of the periodic orbits, it is also possible to
assume that $f_\ph(0,\ph)=1/T_+$ and $f_\ph(1/2,\ph)=1/T_-$ for all $\ph\in\R$,
where $T_\pm$ denote the periods of the unstable and stable
orbit~\cite[Proposition~2.1]{BG_periodic2}\footnote{Because of second-order
terms in It\^o's formula, the periodic orbits of the reparametrized system may
not lie exactly on horizontal lines $r=n/2$, but be shifted by a small amount of
order $\sigma^2$.}. The instability of the orbit $r=0$ means that the
characteristic exponent 
\begin{equation}
 \label{exit04}
 \lambda_+ = 
 \int_0^1 \partial_r f_r(0,\ph)\6\ph
\end{equation} 
is strictly positive. The similarly defined exponent $-\lambda_-$ of the stable
orbit is negative. It is then possible to redefine $r$ in such a way that 
\begin{align}
 \nonumber
 f(r,\ph) &= \lambda_+ r + \Order{r^2}\;, \\
 f(r,\ph) &= -\lambda_-(r-1/2) + \Order{(r-1/2)^2} 
 \label{exit05}
\end{align}
for all $\ph\in\R$ (see again~\cite[Proposition~2.1]{BG_periodic2}). It will be
convenient to assume that $f_\ph(r,\ph)$ is positive, bounded away from zero,
for all $(r,\ph)$. 

Finally, for definiteness, we assume that the system is asymmetric, in such a
way that it is easier for the system starting with $r$ near $-1/2$ to reach the
unstable orbit in $r=0$ rather than its translate in $r=-1$. This corresponds
intuitively to the potential in~\figref{fig_phase_slip} tilting to the right,
and can be formulated precisely in terms of large-deviation rate functions
introduced in Section~\ref{ssec_ldp} below. 


\subsection{The harmonic measure}
\label{ssec_hm}

Fix an initial condition $(r_0\in(-1,0),\ph_0=0)$ and let 
\begin{equation}
 \label{hm01}
 \tau_0 = \inf\setsuch{t>0}{r_t = 0}
\end{equation} 
denote the first-hitting time of the unstable orbit. Note that $\tau_0$ can also
be viewed as the first-exit time from the set $\cD=\set{r<0}$. The crossing
phase $\ph_{\tau_0}$ is equal to the exit location from $\cD$, and its
distribution
is also known as the \emph{harmonic measure} associated with the
infinitesimal generator 
\begin{equation}
 \label{hm02} 
 L = \sum_{i\in\set{r,\ph}} f_i(x) \dpar{}{x_i} + 
\frac{\sigma^2}{2} \sum_{i,j\in\set{r,\ph}}D_{ij}(x)
\dpar{^2}{x_i\partial x_j}
\end{equation}
of the diffusion process. It is known that the harmonic measure admits a smooth
density for sufficiently smooth $f$, $g$ and $\partial\cD$
\cite{BenArous_Kusuoka_Stroock_1984}.  

It follows from Dynkin's formula~\cite[Section 7.4]{Oksendal} that for any
continuous bounded test function $b:\partial\cD\to\R$, the function
$h(x)=\expecin{x}{b(\ph_{\tau_0})}$ satisfies the boundary value
problem\footnote{Several tools will require $\cD$ to be a bounded set. This does
not create any problems, because our assumptions on the deterministic vector
field imply that probabilities are only affected by a negligible amount if
$\cD$ is replaced by its intersection with some large compact set.}
\begin{alignat}{3}
\nonumber
 Lh(x) &= 0 	      &x\in&\cD\;, \\
 h(x)  &= b(x) 	\qquad &x\in&\partial\cD\;.
 \label{hm03}
\end{alignat}
One may think of the case of a sequence $b_n$ converging to the
indicator function $1_{\set{\ph\in B}}$. Then the associated $h_n$ converge to
$h(x)=\probin{x}{\ph_{\tau_0}\in B}$, giving the harmonic measure of
$B\subset\partial\cD$.  While it is
in general difficult to solve the equation~\eqref{hm03} explicitly, the fact
that $Lh=0$ ($h$ is said to be \emph{harmonic}) yields some useful information.
In particular, $h$ satisfies a maximum principle and Harnack
inequalities~\cite[Chapter~9]{Gilbarg_Trudinger}.


\subsection{Large deviations}
\label{ssec_ldp}

The theory of large deviations has been developed in the context of general SDEs
of the form~\eqref{exit01} by Freidlin and Wentzell~\cite{FW}. With a path
$\gamma:[0,T]\to\R\!^2$ it associates the \emph{rate function} 
\begin{equation}
 \label{ldp01}
 I_{[0,T]}(\gamma) = \frac12 \int_0^T (\dot\gamma_s-f(\gamma_s))^\text{T} 
 D(\gamma_s)^{-1} (\dot\gamma_s-f(\gamma_s)) \6s\;.
\end{equation} 
Roughly speaking, the probability of the stochastic process tracking a
particular path $\gamma$ on $[0,T]$ behaves like
$\e^{-I_{[0,T]}(\gamma)/\sigma^2}$ as $\sigma\to0$. 

In the case of the stochastic exit problem from a domain $\cD$, containing a
unique attractor\footnote{In the present context, an attractor $\cA$ is an 
equivalence set for the equivalence relation $\sim_\cD$ on $\cD$, defined by 
$x\sim_\cD y$ whenever one can find a $T>0$ and a path $\gamma$ connecting
$x$ and $y$ in time $T$ and staying in $\cD$ such that $I_{[0,T]}(\gamma)=0$,
cf.~\cite[Section~6.1]{FW}. In addition, $\cD$ should belong to the basin of
attraction of $\cA$. In other words, deterministic orbits starting in
$\cD$ should converge to $\cA$, and the set $\cA$ should have no proper subsets
invariant under the deterministic flow.} $\cA$, the theory of large deviations
yields in particular the following information. For $y\in\partial\cD$ let 
\begin{equation}
 \label{ldp02}
 V(y) = \inf_{T>0}\inf_{\gamma\colon \cA\to y} I_{[0,T]}(\gamma)\;,
\end{equation} 
be the \emph{quasipotential}, where the second infimum runs over all paths
connecting $\cA$ to $y$ in time $T$. Then for $x_0\in\cA$ 
\begin{equation}
 \label{ldp03}
 \lim_{\sigma\to0} \sigma^2 \log\bigexpecin{x_0}{\tau_0} =
\inf_{y\in\partial\cD} V(y)\;.
\end{equation} 
Furthermore, if the quasipotential reaches its infimum at a unique isolated
point $y^*\in\partial\cD$, then 
\begin{equation}
 \label{ldp04}
 \lim_{\sigma\to0} \bigprobin{x_0}{\norm{x_{\tau_0} - y^*} > \delta} = 0
\end{equation} 
for all $\delta>0$. This means that exit locations concentrate in points where
the quasipotential is minimal. 

If we try to apply this last result to our problem, however, we realise that it
does not give any useful information. Indeed, the quasipotential $V$ is
constant on the unstable orbit $\set{r=0}$, because any two points on the orbit
can be connected at zero cost, just by tracking the orbit. 

Nevertheless, the theory of large deviations provides some useful information,
since it allows to determine most probable exit paths. The rate
function~\eqref{ldp01} can be viewed as a Lagrangian action. Minimizing the
action via Euler--Lagrange equations is equivalent to solving Hamilton
equations with Hamiltonian 
\begin{equation}
 \label{ldp05}
 H(\gamma,\eta) = \frac12 \eta^\text{T} D(\gamma) \eta + f(\gamma)^\text{T}
\eta\;,
\end{equation}
where $\eta=D(\gamma)^{-1}(\dot\gamma - f(\gamma))=(p_r,p_\ph)$ is the moment
conjugated to $\gamma$. This is a two-degrees-of-freedom Hamiltonian, whose
orbits live in a four-dimensional space, which is, however, foliated into
three-dimensional hypersurfaces of constant $H$. 

Writing out the Hamilton equations (cf.~\cite[Section~2.2]{BG_periodic2}) shows
that the plane $\set{p_r=p_\ph=0}$ is invariant. It corresponds to deterministic
motion, and contains in particular the periodic orbits of the original system. 
These turn out to be hyperbolic periodic orbits of the three-dimensional
flow on the set $\set{H=0}$, with characteristic exponents $\pm\lambda_+T_+$ and
$\mp\lambda_-T_-$. Typically, the unstable manifold of the stable orbit and
the stable manifold of the unstable orbit will intersect transversally, and the
intersection will correspond to a minimiser $\gamma_\infty$ of the rate
function, connecting the two orbits in infinite time. In the sequel, we will
assume that this is the case, and that $\gamma_\infty$ is unique up to shifts
$\ph\mapsto\ph+n$ (cf.~\cite[Assumption~2.3 and Figure 2.2]{BG_periodic2}
and~\figref{fig_ldp_section}). 

\begin{example}
\label{ex:Melnikov}
Assume that in~\eqref{exit03}, $f_r(r,\ph)=\sin(2\pi r)[1+\eps\sin(2\pi
r)\cos(2\pi\ph)]$, whereas $f_\ph(r,\ph)=\omega$ and $g_r=g_\ph=1$. The
resulting Hamiltonian takes the form 
\begin{equation}
 \label{ldp06}
 H(r,\ph,p_r,p_\ph) = \frac12(p_r^2+p_\ph^2) + \sin(2\pi r)[1+\eps\sin(2\pi
r)\cos(2\pi\ph)] p_r + \omega p_\ph\;.
\end{equation} 
In the limiting case $\eps=0$, the system is invariant under shifts along
$\ph$, and thus $p_\ph$ is a first integral. The unstable and stable manifolds
of the periodic orbits do not intersect transversally. In fact, they are
identical, and given by the equation $p_r=-2\sin(2\pi r)$, $p_\ph=0$. However,
for small positive $\eps$, Melnikov's method~\cite[Chapter~6]{GH} allows to
prove that the two manifolds intersect transversally. 
\end{example}


\subsection{Random Poincar\'e maps}
\label{ssec_rpm}

\begin{figure}
\centerline{\includegraphics*[clip=true,height=55mm]{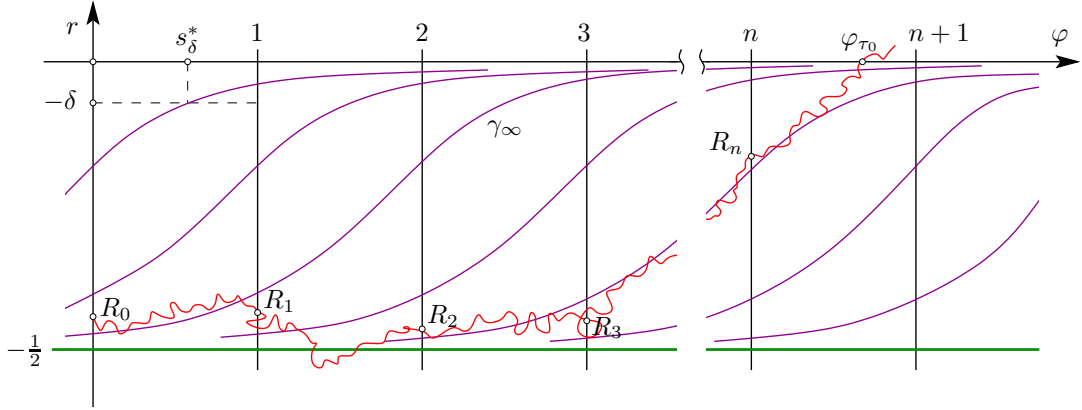}}
 \figtext{ 
	\writefig	0.9	5.5	$r$
	\writefig	3.35	5.4	$1$
	\writefig	5.55	5.4	$2$
	\writefig	7.7	5.4	$3$
	\writefig	9.9	5.4	$n$
	\writefig	11.2	5.35	$\ph_{\tau_0}$
	\writefig	12.1	5.4	$n+1$
	\writefig	14.0	5.4	$\ph$
	\writefig	6.5	4.2	$\gamma_\infty$
	\writefig	1.35	1.75	$R_0$
	\writefig	3.525	1.85	$R_1$
	\writefig	5.7	1.65	$R_2$
	\writefig	7.88	1.5	$R_3$
	\writefig	9.45	3.95	$R_n$
	\writefig	0.1	1.2	$-\frac12$
	\writefig	0.6	4.5	$-\delta$
	\writefig	2.35	5.35	$s^*_\delta$
 }
\caption[]{Definition of the random Poincar\'e map. The sequence
$(R_0,R_1,\dots,R_{\intpart{\tau_0}})$ forms a Markov chain, killed at
$\intpart{\tau_0}$, where $\tau_0$ is the first time the process hits the
unstable periodic orbit in $r=0$. The process is likely to track a translate of
the path $\gamma_\infty$ minimizing the rate function. 
}
\label{fig_random_poincare}
\end{figure}

The periodicity in $\ph$ of the system~\eqref{exit03} yields useful information
on the distribution of the crossing phase $\ph_{\tau_0}$. We fix an initial
condition $(r_0,\ph_0=0)$ with $-1<r_0<0$ and define for every $n\in\N$  
\begin{equation}
 \label{rpm01}
 \tau_n = \inf\setsuch{t>0}{\ph_t = n}\;.
\end{equation} 
In addition, we kill the process at the first time $\tau_0$ it hits the unstable
orbit at $r=0$, and set $\tau_n=\infty$ whenever $\ph_{\tau_0} < n$. The
sequence
$(R_0,R_1,\dots,R_N)$ defined by $R_k=r_{\tau_k}$ and $N = \intpart{\tau_0}$
defines a substochastic Markov chain on $E=\R_-$, which records the successive
values of $r$ whenever $\ph$ reaches for the first time the vertical lines
$\set{\ph=k}$ (\figref{fig_random_poincare}). This Markov chain has a transition
kernel with density $k(x,y)$,
that is, 
\begin{equation}
 \label{rpm02}
  \pcond{R_{n+1}\in B}{r_{\tau_n}=R_n} =: K(R_n,B) = \int_B k(R_n,y)\6y\;,
  \qquad B\subset E
\end{equation} 
for all $n\geqs0$. In fact, $k(x,y)$ is obtained\footnote{Again, for technical
reasons, one has to replace the set $\set{\ph<1,r<0}$ by a large bounded set,
but this modifies probabilities by exponentially small errors that will be
negligible.} by restricting to $\set{\ph=1}$ the harmonic measure for exit from
$\set{\ph<1,r<0}$, for a starting point $(0,x)$. We denote by $K^n$ the $n$-step
transition probabilities defined recursively by 
\begin{equation}
 \label{rpm03}
 K^n(R_0,B) := \probin{R_0}{R_{n}\in B} 
 = \int_E K^{n-1}(R_0,\6y)K(y,B)\;.
\end{equation} 
If we decompose $\ph=n+s$ into its integer part $n$ and fractional part $s$, we
can write 
\begin{equation}
 \label{rpm04}
 \probin{0,R_0}{\ph_{\tau_0}\in n+\6s} 
 = \int_E K^n(R_0,\6y) \probin{0,y}{\ph_{\tau_0}\in\6s}\;.
\end{equation} 
Results by Fredholm~\cite{Fredholm_1903} and Jentzsch~\cite{Jentzsch1912},
extending the well-known Perron--Frobenius theorem, show that $K$ admits a
spectral decomposition. In particular, $K$ admits a \emph{principal eigenvalue}
$\lambda_0$, which is real, positive, and larger than the module of all other
eigenvalues $\lambda_i$. The substochastic nature of the Markov chain, due to
the killing, implies that $\lambda_0<1$. If we can obtain a bound $\rho<1$ on
the ratio $\abs{\lambda_i}/\lambda_0$ valid for all $i\geqs1$ (\emph{spectral
gap} estimate), then we can write 
\begin{equation}
 \label{rpm05}
 K^n(R_0,B) = \lambda_0^n \pi_0(B) \bigl[ 1+\Order{\rho^n} \bigr]
\end{equation} 
as $n\to\infty$. Here $\pi_0$ the probability measure defined by the right
eigenfunction of $K$ corresponding to
$\lambda_0$~\cite{Jentzsch1912,KreinRutman1948,Birkhoff1957}. Since 
\begin{equation}
 \label{rpm06}
 \pcondin{R_0}{R_n\in B}{N>n} 
 = \frac{K^n(R_0,B)}{K^n(R_0,E)} = \pi_0(B) \bigl[ 1+\Order{\rho^n} \bigr]\;,
\end{equation}
the measure $\pi_0$ represents the asymptotic probability distribution of the
process conditioned on having survived. It is called the \emph{quasistationary
distribution} of the process~\cite{Yaglom56,Seneta_VereJones_1966}.
Plugging~\eqref{rpm05} into~\eqref{rpm04}, we see that 
\begin{equation}
 \label{rpm07}
 \probin{0,R_0}{\ph_{\tau_0}\in n+\6s} 
 = \lambda_0^n \int_E \pi_0(\6y) \probin{0,y}{\ph_{\tau_0}\in\6s}
 \bigl[ 1+\Order{\rho^n} \bigr]\;.
\end{equation} 
This implies that the distribution of crossing phases $\ph_{\tau_0}$
asymptotically behaves like a periodically modulated geometric distribution: its
density $P$ satisfies $P(\ph+1) = \lambda_0 P(\ph)$ for large $\ph$. 


\section{Log-periodic oscillations}
\label{sec_logper}

In this section, we formulate our main result on the distribution of crossing
phases $\ph_{\tau_0}$ of the unstable orbit, which describe the position of
phase
slips. This result is based on the work~\cite{BG_periodic2}, but we will
reformulate it in order to allow comparison with related results. 


\subsection{The distribution of crossing phases}
\label{ssec_cycling}

Before stating the results applying to general nonlinear equations of the
form~\eqref{exit03}, let us consider a system approximating it near the 
unstable orbit at $r=0$, given by 
\begin{align}
 \nonumber
 \6r_t &= \lambda_+ r_t \6t + \sigma g_r(0,\ph_t)\6W_t\;, \\
 \6\ph_t &= \frac{1}{T_+}\6t\;. 
 \label{cycling01}
\end{align}
This system can be transformed into a simpler form by combining a
$\ph$-dependent scaling and a random time change. Indeed, let 
$\hper(\ph)$ be the periodic solution of 
\begin{equation}
 \label{cycling03}
 \frac{\6h}{\6\ph} = 2\lambda_+T_+ h - D_{rr}(0,\ph)\;, 
\end{equation} 
and set $r=[2\lambda_+T_+\hper(\ph)]^{1/2}y$. Then  It\^o's formula yields 
\begin{equation}
 \label{cycling03b}
 \6y_t = \frac{D_{rr}(0,\ph_t)}{2T_+\hper(\ph_t)} y_t \6t 
 + \sigma \frac{g_r(0,\ph_t)}{\sqrt{2\lambda_+T_+\hper(\ph_t)}}\6W_t\;.
\end{equation} 
Next, we introduce the function 
\begin{equation}
 \label{cycling05}
 \theta(\ph) = \lambda_+T_+\ph - \frac12 \log \biggl(
\frac{\hper(\ph)}{2\hper(0)^2}\biggr)\;,
\end{equation} 
which should be thought of as a parametrisation of the unstable orbit that makes
the stochastic dynamics as simple as possible. Indeed, note that
$\theta(\ph+1)=\theta(\ph)+\lambda_+T_+$ and 
$\theta'(\ph) = D_{rr}(0,\ph)/(2\hper(\ph)) > 0$. Thus the random time change 
$\6t = [\lambda_+T_+/\theta'(\ph_t)]\6s$ yields the equation 
\begin{equation}
 \6y_s = \lambda_+ y_s \6s + \sigma \tilde g_r(0,s)\6W_s\;,
 \qquad \tilde g_r(0,s) = \frac{g_r(0,\ph_t)}{\sqrt{D_{rr}(0,\ph_t)}}\;,
 \label{cycling02}
\end{equation}
in which the effective noise intensity is constant, i.e.\   
$\widetilde D_{rr}(0,s) = \tilde g_r(0,s)\tilde g_r(0,s)^\text{T}=1$.

In order to formulate the main result of this section, we set 
\begin{equation}
 \label{cycling06}
 \theta_\delta(\ph) = \theta(\ph) - \log\delta 
 + \log \biggl( \frac{\hper(s^*_\delta)}{\hper(0)}\biggr)\;,
\end{equation} 
where $s^*_\delta\in[0,1)$ is such that $(s^*_\delta,-\delta)$ belongs to a
translate of the optimal path $\gamma_\infty$ (\figref{fig_random_poincare}). 
A real-valued random variable $Z$ is said to follow the \emph{standard Gumbel
law} if 
\begin{equation}
 \label{cycling07}
 \bigprob{Z\leqs t} = \e^{-\e^{-t}}
 \qquad \forall t\in\R\;.
\end{equation} 
\figref{fig_Gumbel} shows the density $\e^{-t-\e^{-t}}$ of a standard Gumbel
law. 

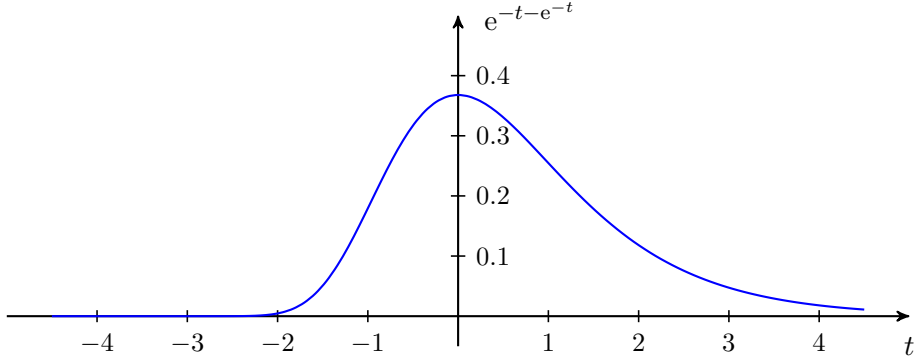
\begin{figure}
\begin{center}
\begin{tikzpicture}[>=stealth',main node/.style={circle,minimum
size=0.25cm,fill=blue!20,draw},x=1.2cm,y=0.8cm]

\draw[->,thick] (-5,0) -> (5,0);
\draw[->,thick] (0,-0.5) -> (0,5);

\foreach \x in {1,...,4}
\draw[semithick] (\x,-0.1) -- node[below=0.1cm] {{\small $\x$}}
(\x,0.1);

\foreach \x in {-1,...,-4}
\draw[semithick] (\x,-0.1) -- node[below=0.1cm] {{\small $\x$}}
(\x,0.1);

\foreach \y in {1,...,4}
\draw[semithick] (-0.08,\y) -- node[right=0.1cm] {{\small $0.\y$}}
(0.08,\y);

\draw[blue,thick,-,smooth,domain=-4.5:4.5,samples=75,/pgf/fpu,/pgf/fpu/output
format=fixed] plot (\x, {
10*exp(-\x -exp(-\x) ) });

\node[] at (5.0,-0.5) {$t$};
\node[] at (0.8,5.0) {$\e^{-t-\e^{-t}}$};
\end{tikzpicture}
\vspace{-5mm}
\end{center}
\caption[]{Density of a standard Gumbel random variable.
}
\label{fig_Gumbel}
\end{figure}

\begin{theorem}[{\cite[Theorem~2.4]{BG_periodic2}}]
\label{thm_periodic2} 
Fix an initial condition $(r_0,\ph_0=0)$ of the nonlinear
system~\eqref{exit03} with $r_0$ sufficiently close to
the stable orbit in $r=-1/2$. There exist $\beta, c>0$ such that for any
sufficiently small $\delta,\Delta>0$, there exists $\sigma_0>0$ such that for
$0<\sigma<\sigma_0$, 
\begin{align}
\nonumber
\biggprobin{r_0,0}{\frac{\theta_\delta(\ph_{\tau_0})}{\lambda_+T_+} \in
[t,t+\Delta]} 
={}& \Delta [1-\lambda_0(\sigma)]\lambda_0(\sigma)^t 
Q_{\lambda_+T_+} \biggl( \frac{\abs{\log\sigma}}{\lambda_+T_+} - t +
\Order{\delta}\biggr) \\
&{}\times \biggl[ 1 + \Order{\e^{-c\ph/\abs{\log\sigma}}} + 
\Order{\delta\abs{\log\delta}} + \Order{\Delta^\beta}\biggr]\;.
\label{cycling08} 
\end{align}
Here $\lambda_0(\sigma)$ is the principal eigenvalue of the Markov chain, and
$1-\lambda_0(\sigma)$ is of order $\e^{-I_\infty/\sigma^2}$, where
$I_\infty=I(\gamma_\infty)$ is the value of the rate function for the path
$\gamma_\infty$. Furthermore, $Q_{\lambda_+T_+}(x)$ is the periodic function,
with period $1$, given by 
\begin{equation}
 \label{cycling08A}
 Q_{\lambda_+T_+}(x) = \sum_{n\in\Z} A\bigl( \lambda_+T_+(n-x) \bigr)\;,
\end{equation}
where 
\begin{equation}
 \label{cycling09}
 A(x) = \exp \Bigl\{ -2x - \frac12 \e^{-2x} \Bigr\}
\end{equation} 
is the density of $(Z-\log2)/2$, with $Z$ a standard Gumbel variable. 
\end{theorem}

We will discuss various implications of this result in the next sections. The
periodic dependence on $\log\sigma$ will be addressed in Section~\ref{ssec_osc},
and we will say more on the Gumbel law in Section~\ref{sec_Gumbel}. For now, let
us give a reformulation of the theorem, which will ease comparison with other
related results. Following~\cite{HitczenkoMedvedev,HitczenkoMedvedev1}, we say
that an integer-valued random variable $Y$ is \emph{asymptotically geometric}
with success probability $p$ if 
\begin{equation}
 \label{cycling10}
 \lim_{n\to\infty} \pcond{Y=n+1}{Y>n} = p\;.
\end{equation} 
We use $\lim_{n\to\infty}\Law(X_n)=\Law(X)$ to denote
convergence in distribution of a sequence of random variables $X_n$ to a
random variable $X$. 

\begin{theorem}
\label{thm_convergence_Gumbel} 
There exists a family 
$(Y_m^\sigma)_{m\in\N,\sigma>0}$ of asymptotically geometric random
variables such that
\begin{equation}
 \label{cycling11}
 \lim_{m\to\infty} \Bigl[
 \lim_{\sigma\to0} \Law \bigl( \theta(\ph_{\tau_0}) - \abs{\log\sigma} -
\lambda_+T_+ Y_m^\sigma \bigr)
 \Bigr]
 = \Law \biggl( \frac{Z}{2} - \frac{\log2}{2}\biggr) \;,
\end{equation} 
where $Z$ is a standard Gumbel random variable independent of the
$Y_m^\sigma$. The success probability of $Y_m^\sigma$ is of the form 
$p_{m,\sigma} = \e^{-I_m/\sigma^2}$, where $I_m = I_\infty +
\Order{\e^{-2m\lambda_+T_+}}$. 
\end{theorem}

This theorem is almost a corollary of Theorem~\ref{thm_periodic2}, but a little
work is required to control the limit $m\to\infty$, which corresponds to the
limit $\delta\to0$. We give the details in Appendix~\ref{sec_proof_Gumbel}. 

The interpretation of~\eqref{cycling11} is as follows. To reach the unstable
orbit at $\ph_{\tau_0}$, the system will track, with high probability, a
translate $\gamma_\infty(\cdot+n)$ of the optimal path $\gamma_\infty$
(\figref{fig_random_poincare}). The random variable $Y_m^\sigma$ is the index
$n$ of the chosen translate. This index follows an approximately geometric
distribution of parameter $1-\lambda_0(\sigma) \simeq \e^{-I_\infty/\sigma^2}$,
which also manifests itself in the factor $(1-\lambda_0)\lambda_0^t$
in~\eqref{cycling08}. The distribution of $\ph_{\tau_0}$ conditional on the
event $\set{Y_m^\sigma=n}$ converges to a shifted Gumbel distribution --- we
will come back to this point in Section~\ref{sec_Gumbel}. 

We may not be interested in the integer part of the crossing phase
$\ph_{\tau_0}$,
which counts the number of rotations around the orbit, but only in its
fractional part $\hat\ph_{\tau_0}=\ph_{\tau_0}\pmod{1}$. Then it follows
immediately
from~\eqref{cycling11} and the fact that $Y_m^\sigma$ is integer-valued that 
\begin{equation}
 \label{cycling12}
 \lim_{\sigma\to0} \Law \bigl( \theta(\hat\ph_{\tau_0}) - \abs{\log\sigma}
\bigr)
 = \Law \biggl( 
 \biggl[\frac{Z}{2} - \frac{\log2}{2}\biggr] \pmod{\lambda_+T_+}\biggr) \;.
\end{equation} 
The random variable on the right-hand side has a density given
by~\eqref{cycling08A}. This result can also be derived directly
from~\cite[Corollary~2.5]{BG_periodic2}, by the same procedure as the one used
in Section~\ref{ssec_pG_u}. 

\begin{remark} \hfill
\begin{enum}
\item 	The index $m$ in~\eqref{cycling11} seems artificial, and one would
like to have a similar result for the law of
$\theta(\ph_{\tau_0})-\abs{\log\sigma}-\lambda_+T_+Y_\infty^\sigma$.
Unfortunately,
the convergence as $\sigma\to0$ is not uniform in $m$, so that the two limits
in~\eqref{cycling11} have to be taken in that particular order.

\item 	The speed of convergence in~\eqref{cycling10} depends on the spectral
gap of the Markov chain. In~\cite[Theorem~6.14]{BG_periodic2}, we proved that
this spectral gap is bounded by $\e^{-c/\abs{\log\sigma}}$ for some constant
$c>0$, though we expect that the gap can be bounded uniformly in $\sigma$. We
expect, but have not proved, that the constant $c$ is uniform in $m$ (i.e.\
uniform in the parameter $\delta$). 
\end{enum}
\end{remark}


\subsection{The origin of oscillations}
\label{ssec_osc}

A striking aspect of the expression~\eqref{cycling08} for the distribution of
$\ph_{\tau_0}$ is that it depends periodically on $\abs{\log\sigma}$. This means
that as $\sigma\to0$, the distribution does not converge, but is endlessly
shifted around the unstable orbit proportionally to $\abs{\log\sigma}$. This
phenomenon has been discovered by Martin Day, who called it \emph{cycling}
\cite{Day7,Day3,Day6,Day4}. See
also~\cite{MS4,BG7,Getfert_Reimann_2009,Getfert_Reimann_2010} for related work. 

The intuitive explanation of cycling is as follows. We have seen that the
large-deviation rate function is minimized by a path $\gamma_\infty$ (and its
translates). The path approaches the unstable orbit as $\ph\to\infty$, and the
stable one as $\ph\to-\infty$. The distance between $\gamma_\infty$ and the
unstable orbit satisfies 
\begin{equation}
 \label{osc01}
 |r(\ph)| \simeq c\e^{-\theta(\ph)}
 \qquad
 \text{as } \ph\to\infty\;.
\end{equation} 
This implies 
\begin{equation}
 \label{osc02}
 |r(\ph)| = \sigma 
 \quad \Leftrightarrow \quad
 \theta(\ph) \simeq |\log\sigma| + \log c\;.
\end{equation} 
Thus everything behaves as if the unstable orbit has an \lq\lq effective
thickness\rq\rq\ equal to the standard deviation $\sigma$ of the noise. Escape
becomes likely once the optimal path $\gamma_\infty$ touches the thickened
unstable orbit.

It is interesting to note that the periodic dependence on the logarithm of a
parameter, or \emph{log-periodic oscillations}, appear in many systems
presenting discrete-scale invariance~\cite{Sornette_98}. 
These include for instance 
hierarchical models in
statistical physics \cite{Derrida_Itzykson_Luck_84,Costin_Giacomin_2013}
and self-similar networks \cite{Doucot_etal_PRL_86},
diffusion through fractals \cite{Akkermans_etal_EPL_2009,Dunne_JPA_2012}
and iterated maps \cite{deMoura_etal_PRE_2000,Derrida_Giacomin_2014}. 
One link with the present situation is that~\eqref{osc01} implies a
discrete-scale invariance, since scaling $r$ by a factor 
$\e^{-\theta(1)}=\e^{-\lambda_+T_+}$ is equivalent to scaling the noise
intensity by the same factor. There might be deeper connections due to the fact
that certain key functions, as the Gumbel distribution in our case, obey
functional equations --- see for instance the similar behaviour
of the example in~\cite[Remark~3.1]{Derrida_Giacomin_2014}. 

\begin{remark}
\label{rem_osc} 
The periodic \lq\lq cycling profile\rq\rq\ $Q_{\lambda_+T_+}$ admits the
Fourier series representation 
\begin{equation}
 \label{osc03}
 Q_{\lambda_+T_+}(x) = \sum_{k\in\Z} a_k \e^{2\pi\icx k x}\;, 
 \qquad
 a_k = \frac{2^{-\pi\icx k/(\lambda_+T_+)}}{\lambda_+T_+} 
 \Gamma \biggl( 1 - \frac{\pi\icx k}{\lambda_+T_+}\biggr)\;,
\end{equation} 
where $\Gamma$ is Euler's Gamma function. $Q_{\lambda_+T_+}$ is also an
\emph{elliptic function}, since in addition to being periodic in the real
direction, it is also periodic in the imaginary direction. Indeed, by definition
of $A(x)$, we have 
\begin{equation}
 \label{osc04}
 Q_{\lambda_+T_+} \biggl( z + \frac{\pi\icx}{\lambda_+T_+}\biggr)
 = Q_{\lambda_+T_+}(z) 
 \qquad \forall z\in\C\;.
\end{equation} 
Being non-constant and doubly periodic, $Q_{\lambda_+T_+}$ necessarily admits
at least two poles in every translate of the unit cell 
$(0,1)\times(0,\pi\icx/(\lambda_+T_+))$.  
\end{remark}


\section{The Gumbel distribution}
\label{sec_Gumbel}


\subsection{Extreme-value theory}
\label{ssec_evt}

Let $X_1,X_2,\dots$ be a sequence of independent, identically
distributed (i.i.d.) real random variables, with common distribution function 
$F(x)=\prob{X_1\leqs x}$. Extreme-value theory is concerned with deriving the
law of the maximum 
\begin{equation}
 \label{evt01}
 M_n = \max\set{X_1,\dots,X_n}
\end{equation} 
as $n\to\infty$. It is immediate to see that the distribution function of $M_n$
is $F(x)^n$. We will say that $F$ belongs to the domain of attraction of a
distribution function $\Phi$, and write $F\in D(\Phi)$, if there exist
sequences of real numbers $a_n>0$ and $b_n$ such that
\begin{equation}
 \label{evt02}
 \lim_{n\to\infty} F(a_nx+b_n)^n = \Phi(x)
 \qquad \forall x\in\R\;.
\end{equation} 
This is equivalent to the sequence of random variables $(M_n-b_n)/a_n$
converging in distribution to a random variable with distribution function 
$\Phi$. Clearly, if $F\in D(\Phi)$, then one also has $F\in D(\Phi(ax+b))$
for all $a>0,b\in\R$, so it makes sense to work with equivalence classes
$\set{\Phi(ax+b)}_{a,b}$. 

Any possible limit of~\eqref{evt02} should satisfy the functional
equation 
\begin{equation}
 \label{evt03}
 \Phi(ax+b)^2 = \Phi(x)
 \qquad \forall x\in\R
\end{equation} 
for some constants $a>0, b\in\R$.
Fr\'echet~\cite{Frechet_1927}, Fischer and Tippett~\cite{Fisher_Tippett_1928}
and Gnedenko~\cite{Gnedenko_1943} have shown that if one excludes the degenerate
case $F(x)=1_{\set{x\geqs c}}$, then the only possible solutions
of~\eqref{evt03} are in one of the following three classes, where $\alpha>0$ is
a parameter:
\begin{alignat}{3}
\nonumber
\Phi_\alpha(x) &= \e^{-x^{-\alpha}} 1_{\set{x>0}} 
& \qquad &\text{Fr\'echet law\;,} \\
\nonumber
\Psi_\alpha(x) &= \e^{-(-x)^{\alpha}} 1_{\set{x\leqs 0}} + 1_{\set{x>0}}
& &\text{(reversed) Weibull law\;,} \\
\Lambda(x) &= \e^{-\e^{-x}}
& &\text{Gumbel law\;.}
 \label{evt04}
\end{alignat}
 
In~\cite{Gnedenko_1943}, Gnedenko gives precise characterizations on when $F$
belongs to the domain of attraction of each of the above laws. Of particular
interest to us is the following result. Let 
\begin{equation}
 \label{evt05}
 R(x) = 1-F(x) = \prob{X_1 > x}
\end{equation} 
denote the tail probabilities of the i.i.d.\ random variables $X_i$. 

\begin{lemma}[{\cite[Lemma~4]{Gnedenko_1943}}]
\label{lem_Gendenko_4}
A nondegenerate distribution function $F$ belongs to the domain
of attraction of $\Phi$ if and only if there exist sequences $a_n>0$ and $b_n$
such that 
\begin{equation}
 \label{evt06}
 \lim_{n\to\infty} n R(a_n x + b_n) = -\log\Phi(x)
 \qquad
 \forall x \text{ such that } \Phi(x)>0\;.
\end{equation} 
\end{lemma}

The sequences $a_n$ and $b_n$ are not unique, but
\cite[Theorem~6]{Gnedenko_1943} shows that in the case of the Gumbel
distribution $\Phi=\Lambda$, 
\begin{equation}
 \label{evt07}
 b_n = \inf\biggsetsuch{x}{F(x) > 1-\frac1n} \;, 
 \qquad
 a_n = \inf\biggsetsuch{x}{F(x) + b_n > 1-\frac1{n\e}}
\end{equation} 
is a possible choice. In this way it is easy to check that the normal law is
attracted to the Gumbel distribution.

Another related characterization of $F$ being in the domain of attraction of the
Gumbel law is the following.

\begin{theorem}[{\cite[Theorem~7]{Gnedenko_1943}}]
Let $x_0=\inf\setsuch{x}{F(x)=1}\in\R\cup\set{\infty}$. Then $F\in D(\Lambda)$
if and only if there exists a continuous function $A(z)$ such that 
$\lim_{z\to x_0-}A(z)=0$ and 
\begin{equation}
 \label{evt08}
 \lim_{z\to x_0-} \frac{R(z(1+A(z)x))}{R(z)} 
 = -\log\Lambda(x) = \e^{-x}
 \qquad \forall x\in\R\;.
\end{equation} 
The function $A(z)$ can be chosen such that $A(b_n)=a_n/b_n$ for all $n$, where
$a_n$ and $b_n$ satisfy~\eqref{evt06}.
\end{theorem}

The quantity on the left-hand side of~\eqref{evt08} can be rewritten as 
\begin{equation}
 \label{evt09}
 \bigpcond{X_1 > z(1 + A(z)x)}{X_1>z}\;,
\end{equation} 
that is, it represents a \emph{residual lifetime}. See
also~\cite{Balkema_deHaan_74}. 


\subsection{Length of reactive paths}
\label{ssec_lrp}

The Gumbel distribution also appears in the context of somewhat different exit
problems (which, however, will turn out not to be so different after all).
In~\cite{CerouGuyaderLelievreMalrieu12}, C\'erou, Guyader, Leli\`evre
and Malrieu consider one-dimensional SDEs of the form  
\begin{equation}
\label{lrp01} 
 \6x_t = -V'(x_t)\6t + \sigma\6W_t\;,
\end{equation} 
where $V(x)$ is a double-well potential (\figref{fig_double_well}). 

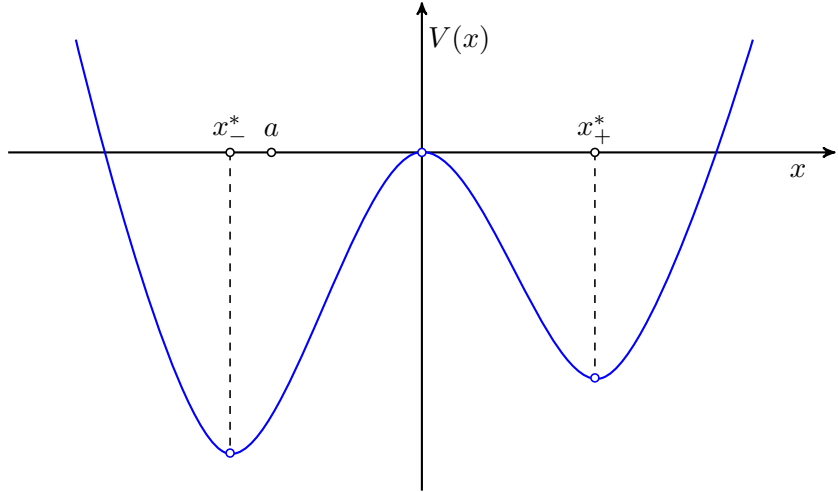
\begin{figure}
\begin{center}
\begin{tikzpicture}[>=stealth',main node/.style={draw,circle,fill=white,minimum
size=3pt,inner sep=0pt}]



\draw[->,thick] (-5.5,0) -> (5.5,0);
\draw[->,thick] (0,-4.5) -> (0,2.0);


\draw[dashed,semithick] (-2.55,0) -- (-2.55,-4);
\draw[dashed,semithick] (2.3,0) -- (2.3,-3);


\draw[blue,thick] plot[smooth,tension=.6]
  coordinates{(-4.6,1.5) (-2.6,-4) (-0.05,0) (2.4,-3) (4.4,1.5)};


\node[main node,blue,fill=white,semithick] at (0,0) {}; 
\node[main node,semithick] at (2.3,0) {}; 
\node[main node,blue,fill=white,semithick] at (2.3,-3) {}; 
\node[main node,semithick] at (-2.55,0) {}; 
\node[main node,blue,fill=white,semithick] at (-2.55,-4) {}; 
\node[main node,semithick] at (-2.0,0) {}; 

\node[] at (-2.55,0.3) {$x^*_-$};
\node[] at (2.3,0.3) {$x^*_+$};
\node[] at (-2,0.3) {$a$};

\node[] at (5.0,-0.25) {$x$};
\node[] at (0.5,1.5) {$V(x)$};

\end{tikzpicture}
\end{center}
\caption[]{An example of double-well potential occurring in~\eqref{lrp01}.
}
\label{fig_double_well}
\end{figure}

Assume, without loss of generality, that the local maximum of $V$ is in $0$.
Denote the local minima of $V$ by $x^*_- < 0 < x^*_+$, and assume
$\lambda=-V''(0)>0$. Pick an initial condition $x_0\in(x^*_-,0)$. A classical
question is to determine the law of the first-hitting time $\tau_b$ of a point
$b\in(0,x^*_+]$. The expected value of $\tau_b$ obeys the so-called
\emph{Eyring--Kramers law}~\cite{Arrhenius,Eyring,Kramers}
\begin{equation}
 \label{lrp02}
 \expecin{x_0}{\tau_b} = 
\frac{2\pi}{\sqrt{V''(x^\star_-)\abs{V''(0)}}}
\e^{2[V(0)-V(x^\star_-)]/\sigma^2} \bigl[ 1 + \Order{\sigma} \bigr]\;.
\end{equation} 
In addition, Day~\cite{Day1} has proved (in a more general context) that the
distribution of $\tau_b$ is asymptotically exponential:
\begin{equation}
 \label{lrp03}
\lim_{\sigma\to0} \bigprobin{x_0}{\tau_b > s \, \expecin{x_0}{\tau_b}} = \e^{-s}
\;.
\end{equation} 
The picture is that sample paths spend an exponentially long time near the local
minimum $x^*_-$, with occasional excursions away from $x^*_-$, until ultimately
managing to cross the saddle. See for instance~\cite{Berglund_irs_MPRF} for a
recent review. 

In transition-path theory~\cite{E_VandenEijnden_06,Vanden_Eijnden_LNP06}, by
contrast, one is interested in the very last bit of
the sample path, between its last visit to $x^*_-$ and its first passage in
$b$. The length of this transition is considerably shorter than $\tau_b$. A way
to formulate this is to fix a point $a\in(x^*_-,x_0)$, and to condition on the
event that the path hits $b$ before hitting $a$. The result can be formulated
as follows (note that our $\sigma$ corresponds to $\sqrt{2\eps}$
in~\cite{CerouGuyaderLelievreMalrieu12}):

\begin{theorem}[{\cite[Theorem~1.4]{CerouGuyaderLelievreMalrieu12}}]
\label{thm_CGLM} 
For any fixed $a < x_0 < 0 < b$ in $(x^*_- ,  x^*_+)$,
\begin{equation}
 \label{lrp04}
  \lim_{\sigma\to0} \Law \bigl( \lambda \tau_b - 2 \abs{\log\sigma} 
  \bigm| \tau_b < \tau_a \bigr)
 = \Law \Bigl( Z + T(x_0,b) \Bigr) \;,
\end{equation} 
where $Z$ is a standard Gumbel variable, and 
\begin{equation}
 \label{lrp05}
 T(x_0,b) = \log \bigl(\abs{x_0}b\lambda \bigr) 
 + \int_{x_0}^0 \biggl( \frac{\lambda}{V'(y)} + \frac{1}{y}\biggr) \6y
 - \int_0^b \biggl( \frac{\lambda}{V'(y)} + \frac{1}{y}\biggr) \6y\;.
\end{equation} 
\end{theorem}

The proof is based on Doob's $h$-transform, which allows to replace the
conditioned problem by an unconditioned one, with a modified drift term. The
new drift term becomes singular as $x\to a_+$. See also~\cite{Lu_Nolen_2014}
for other uses of Doob's $h$-transform in the context of reactive paths.

As shown in~\cite[Section~4]{CerouGuyaderLelievreMalrieu12},
$2\abs{\log\sigma} + T(x_0,b)/\lambda$
is the sum of the deterministic time needed to go from $\sigma$ to $b$, and 
of the deterministic time needed to go from $-\sigma$ to $a$ (in the
one-dimensional setting, paths minimizing the large-deviation rate function are
time-reversed deterministic paths). 


\subsection{Bakhtin's approach}
\label{ssec_Bakhtin}

Yuri Bakhtin has recently provided some interesting insights into the question
of why the Gumbel distribution governs the length of reactive
paths~\cite{Bakhtin_2013a,Bakhtin_2014a}. They apply to linear equations of the
form 
\begin{equation}
\label{Bakhtin01} 
 \6x_t = \lambda x_t\6t + \sigma\6W_t\;,
\end{equation} 
where $\lambda > 0$. 
However, we will see in Section~\ref{sec_slips} below that they can be 
extended to the nonlinear setting by using the technique outlined in
Appendix~\ref{ssec_pG_u}. 

The solution of~\eqref{Bakhtin01} is an \lq\lq explosive Ornstein--Uhlenbeck
process\rq\rq 
\begin{equation}
 \label{Bakhtin02}
 x_t = \e^{\lambda t} \biggl( x_0 + \sigma \int_0^t \e^{-\lambda
s}\6W_s\biggr)\;,
\end{equation} 
which can also be represented in terms of a time-changed Brownian motion, 
\begin{equation}
 \label{Bakhtin03}
 x_t = \e^{\lambda t} \tilde x_t\;,
 \qquad
 \tilde x_t = x_0 + \widetilde W_{\sigma^2(1-\e^{-2\lambda t})/(2\lambda)}
\end{equation} 
(this follows by evaluating the variance of $\tilde x_t$ using It\^o's
isometry). Thus $\tilde x_t - x_0$ is equal in distribution to 
$\sigma\sqrt{(1-\e^{-2\lambda t})/(2\lambda)}\,N$, where $N$ is a standard
normal random variable. 

Assume $x_0<0$ and denote by $\tau_0$ the first-hitting time of $x=0$. Then
Andr\'e's reflection principle allows to write 
\begin{equation}
\label{Bakhtin04} 
\bigpcond{\tau_0 < t}{\tau_0 < \infty}
= \frac{\prob{\tau_0 < t}}{\prob{\tau_0 < \infty}}
= \frac{2\prob{\tilde x_t>0}}{2\prob{\tilde x_\infty>0}}
= \bigpcond{\tilde x_t>0}{\tilde x_\infty>0}\;.
\end{equation}
Now we observe that 
\begin{align}
\nonumber 
\Bigpcond{\tau_0 < t + \frac{1}{\lambda} |\log\sigma|}{\tau_0 < \infty} 
&= \Bigpcond{\tilde x_{t+\frac{1}{\lambda} |\log\sigma|}>0}{\tilde x_\infty>0}
\\
&= \biggpcond{N >
\frac{|x_0|}{\sigma}\sqrt{\frac{2\lambda}{1-\sigma^2\e^{-2\lambda t}}}}
{N > \frac{|x_0|}{\sigma}\sqrt{2\lambda}}\;,
\label{Bakhtin05} 
\end{align}
where $N$ is a standard normal random variable. It follows that 
\begin{equation}
\label{Bakhtin06}
\lim_{\sigma\to0}
\Bigpcond{\tau_0 < t + \frac{1}{\lambda} |\log\sigma|}{\tau_0 < \infty}  
= \exp \bigl\{ -x_0^2 \lambda\e^{-2\lambda t}\bigr\}\;.
\end{equation} 
This can be checked by a direct computation, using tail asymptotics of the
normal law. However, it is more interesting to view the last expression
in~\eqref{Bakhtin05} as a residual lifetime, given by the
expression~\eqref{evt09} with $z=(|x_0|/\sigma)\sqrt{2\lambda}$, $A(z)=z^{-2}$
and $x=x_0^2 \lambda\e^{-2\lambda t}$. The right-hand side of~\eqref{Bakhtin06}
is the distribution function of $(Z+\log(x_0^2\lambda))/(2\lambda)$, where $Z$
is a standard Gumbel variable. Building on this computation, Bakhtin provided a
new proof of the following result, which was already obtained by Day
in~\cite{Day7}. 

\begin{theorem}[{\cite{Day7} and \cite[Theorem~3]{Bakhtin_2014a}}]
\label{thm_Bakhtin_tau0} 
Fix $a<0$ and an initial condition $x_0\in(a,0)$. Then 
\begin{equation}
 \label{Bakhtin07}
  \lim_{\sigma\to0} \Law \bigl( \lambda \tau_0 - \abs{\log\sigma} 
  \bigm| \tau_0 < \tau_a \bigr)
 = \Law \biggl( \frac{Z}{2} + \frac{\log(x_0^2\lambda)}{2}  \biggr) \;. 
\end{equation}
\end{theorem}

Observe the similarity with Theorem~\ref{thm_convergence_Gumbel} (and also with
Proposition~\ref{prop_pG_p1}). The proof in~\cite{Bakhtin_2014a} uses the fact
that conditioning on $\set{\tau_0 < \tau_a}$ is asymptotically equivalent to
conditioning on $\set{\tau_0 < \infty}$. Note that we use a similar argument in
the proof of Theorem~\ref{thm_convergence_Gumbel} in
Appendix~\ref{ssec_pG_proof}. 

The expression~\eqref{Bakhtin07} differs from~\eqref{lrp04} in some factors $2$.
This is due to the fact that Theorem~\ref{thm_Bakhtin_tau0} considers the
first-hitting time $\tau_0$ of the saddle, while Theorem~\ref{thm_CGLM}
considers the first-hitting time $\tau_b$ of a point $b$ in the right-hand
potential well. We will come back to this point in Section~\ref{ssec_tba}. 

The observations presented here provide a connection between first-exit times
and extreme-value theory, via the reflection principle and residual lifetimes. 
As observed in~\cite[Section~4]{Bakhtin_2013a}, the connection depends on the
seemingly accidental property 
\begin{equation}
 \label{Bakhtin08}
 -\log\Lambda(\e^{-x}) = \Lambda(x)\;,
\end{equation} 
or $\Lambda(\e^{-x}) = \e^{-\Lambda(x)}$, of the Gumbel distribution function. 
Indeed, the right-hand side in~\eqref{Bakhtin06} is identified with
$-\log\Lambda(x)$, evaluated in a point $x$ proportional to 
$-\e^{-2\lambda t}$. 


\section{The duration of phase slips}
\label{sec_slips}


\subsection{Leaving the unstable orbit}
\label{ssec_luo}

Consider again, for a moment, the linear equation~\eqref{Bakhtin01}. Now we are
interested in the situation where the process starts in $x_0=0$, and hits a
point $b>0$ before hitting a point $a<0$. In this section, $\Theta$ will denote 
the random variable $\Theta=-\log|N|$, where $N$ is a standard normal variable. 
Its density is given by 
\begin{equation}
 \label{luo01}
 \frac{\6}{\6t} 2 \prob{N < -\e^{-t}}
 = \sqrt{\frac2\pi} \e^{-t-\frac12\e^{-2t}}\;,
\end{equation} 
which is similar to, but different from, the density of a Gumbel
distribution, see~\figref{fig_Theta}. 

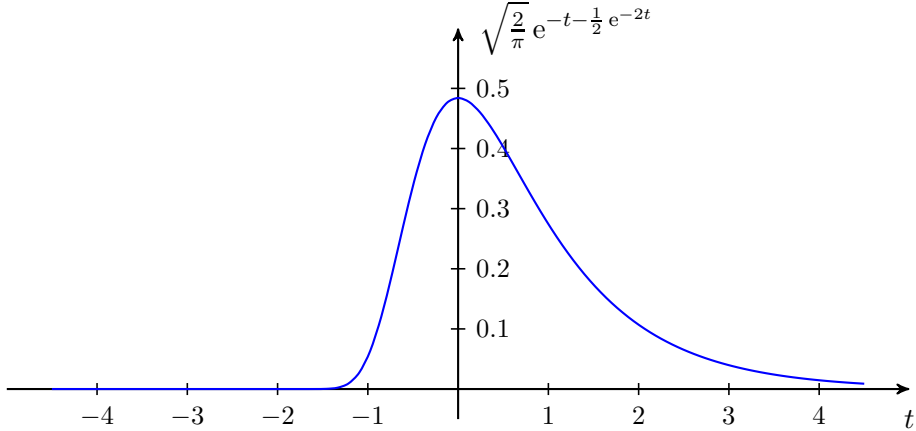
\begin{figure}
\begin{center}
\begin{tikzpicture}[>=stealth',main node/.style={circle,minimum
size=0.25cm,fill=blue!20,draw},x=1.2cm,y=0.8cm]

\draw[->,thick] (-5,0) -> (5,0);
\draw[->,thick] (0,-0.5) -> (0,6.0);

\foreach \x in {1,...,4}
\draw[semithick] (\x,-0.1) -- node[below=0.1cm] {{\small $\x$}}
(\x,0.1);

\foreach \x in {-1,...,-4}
\draw[semithick] (\x,-0.1) -- node[below=0.1cm] {{\small $\x$}}
(\x,0.1);

\foreach \y in {1,...,5}
\draw[semithick] (-0.08,\y) -- node[right=0.1cm] {{\small $0.\y$}}
(0.08,\y);

\draw[blue,thick,-,smooth,domain=-4.5:4.5,samples=75,/pgf/fpu,/pgf/fpu/output
format=fixed] plot (\x, {
sqrt(2/pi)*10*exp(-\x -0.5*exp(-2*\x) ) });

\node[] at (5.0,-0.5) {$t$};
\node[] at (1.2,6.0) {$\sqrt{\frac2\pi}\e^{-t-\frac12\e^{-2t}}$};
\end{tikzpicture}
\vspace{-5mm}
\end{center}
\caption[]{Density of the random variable $\Theta=-\log|N|$.
}
\label{fig_Theta}
\end{figure}

\begin{theorem}[{\cite{Day2,Bakhtin_2008_SPA,Bakhtin_2011_PTRF}}]
\label{thm_Bakhtin_taub} 
Fix $a<0<b$ and an initial condition $x_0=0$. Then the linear
system~\eqref{Bakhtin01} satisfies 
\begin{equation}
 \label{luo02}
  \lim_{\sigma\to0} \Law \bigl( \lambda \tau_b - \abs{\log\sigma} 
  \bigm| \tau_b < \tau_a \bigr)
 = \Law \biggl( \Theta + \frac{\log(2b^2\lambda)}{2}  \biggr) \;. 
\end{equation}
\end{theorem}

The intuition for this result is as follows. Consider first the symmetric case
where $a=-b$ and let $\tau=\inf\setsuch{t>0}{|x_t|=b}=\tau_a\wedge\tau_b$. The
solution of~\eqref{Bakhtin01} starting in $0$ can be written $x_t = \e^{\lambda
t} \tilde x_t$, where $\tilde x_t = \sigma\sqrt{(1-\e^{-2\lambda
t})/(2\lambda)}\,N$ and $N$ is a standard normal random variable,
cf.~\eqref{Bakhtin03}. The condition $|x_\tau|=b$ yields 
\begin{equation}
 \label{luo03}
 b 
 = \e^{\lambda\tau} \sigma \sqrt{\frac{1-\e^{-2\lambda\tau}}{2\lambda}}|N|
 \simeq \e^{\lambda\tau} \sigma
\frac{1}{\sqrt{2\lambda}}|N|\;.
\end{equation}
Solving for $\tau$ yields $\lambda\tau - |\log\sigma| \simeq \log(2\lambda
b^2)/2 - \log|N|$. One can also show~\cite[Theorem~2.1]{Day2} that
$\sign(x_\tau)$ converges to a random variable $\nu$, independent of $N$, such
that $\prob{\nu=1}=\prob{\nu=-1}=1/2$. This implies~\eqref{luo02} in the
symmetric case, and the asymmetric case is dealt with
in~\cite[Theorem~1]{Bakhtin_2008_SPA}. 

Let us return to the nonlinear system~\eqref{exit03} governing the coupled
oscillators. We seek a result similar to Theorem~\ref{thm_Bakhtin_taub} for the
first exit from a neighbourhood of the unstable orbit. The scaling argument
given at the beginning of Section~\ref{ssec_cycling} indicates that simpler
expressions will be obtained if this neighbourhood has a non-constant width of
size proportional to $\sqrt{2\lambda_+T_+\hper(\ph)}$. This is also consistent
with the discussion in~\cite[Section~3.2.1]{BGbook}. Let us thus set 
\begin{equation}
 \label{luo03b}
 \tilde\tau_{\delta} = \inf\Bigsetsuch{t>0}{r_t =
\delta\sqrt{2\lambda_+T_+\hper(\ph)}}\;.
\end{equation} 

\begin{theorem}
\label{thm_exit}
Fix an initial condition $(\ph_0,r_0=0)$ on the
unstable periodic orbit. Then the system~\eqref{exit03} satisfies 
\begin{equation}
 \label{luo04}
  \lim_{\sigma\to0} \Law \bigl( \theta(\ph_{\tilde\tau_{\delta}}) -
\theta(\ph_0) - \abs{\log\sigma} 
  \bigm| \tilde\tau_{\delta} < \tilde\tau_{-\delta} \bigr)
 = \Law \biggl( \Theta + \frac{\log ( 2\lambda_+\delta^2 )}2  + \Order{\delta}
\biggr)
\end{equation}
as $\delta\to0$. 
\end{theorem}

We give the proof in Appendix~\ref{ssec_proof_exit}. Note that this result is
indeed consistent with~\eqref{luo02}, if we take into account the fact that
$\hper(\ph)\equiv1/(2\lambda_+T_+)$ in the case of a constant diffusion matrix
$D_{rr}\equiv 1$. 


\subsection{There and back again}
\label{ssec_tba}

A nice observation in~\cite{Bakhtin_2014a} is that
Theorems~\ref{thm_Bakhtin_tau0} and~\ref{thm_Bakhtin_taub} imply
Theorem~\ref{thm_CGLM} on the length of reactive paths in the linear case. This
follows immediately from the following fact. 

\begin{lemma}[\cite{Bakhtin_2014a}]
\label{lem_Bakhtin} 
Let $Z$ and $\Theta=-\log|N|$ be independent random variables, where $Z$
follows a standard Gumbel law, and $N$ a standard normal law. Then 
\begin{equation}
 \label{tba10}
 \Law \biggl( \frac12 Z + \Theta \biggr) 
 = \Law \biggl( Z + \frac{\log2}2 \biggr) \;.
\end{equation} 
\end{lemma}
\begin{proof}
This follows directly from the expressions 
\begin{equation}
 \label{tba11}
 \bigexpec{\e^{\icx t Z}} = \Gamma(1-\icx t)
 \qquad \text{and} \qquad
 \bigexpec{\e^{\icx t \Theta}} = \bigexpec{|N|^{-\icx\theta}}
 = \frac{2^{-\icx t/2}}{\sqrt{\pi}}\Gamma \biggl( \frac{1-\icx t}{2}\biggr)
\end{equation} 
for the characteristic functions of $Z$ and $\Theta$, 
and the duplication formula for the Gamma function,
$\sqrt{\pi}\,\Gamma(2z)=2^{2z-1} \Gamma(z) \Gamma(z+\tfrac12)$.
\end{proof}

Let us now apply similar ideas to the nonlinear system~\eqref{exit03} in order
to derive information on the duration of phase slips. In order to define this
duration, consider two families of continuous curves $\Gamma^s_-$ and
$\Gamma^s_+$, depending on a parameter $s\in\R$, periodic in the
$\ph$-direction, and such that each $\Gamma^s_-$ lies in the set
$\set{-1/2<r<0}$ and each $\Gamma^s_+$ lies in $\set{0<r<1/2}$. We set  
\begin{equation}
 \label{tba12}
 \tau^s_\pm = \inf\setsuch{t>0}{(r_t,\ph_t)\in\Gamma^s_\pm}\;,
\end{equation} 
while $\tau_0$ is defined as before by~\eqref{hm01}. Given an initial condition
$(r_0=-1/2,\ph_0)$, let us call a \emph{successful phase slip} a sample path
that does not return to the stable orbit $\set{r=-1/2}$ between $\tau^s_-$ and
$\tau_0$, and that does not return to $\Gamma^s_-$ between $\tau_0$ and
$\tau^s_+$ (see~\figref{fig_phase_slip}). Then we have the following result.

\begin{theorem}
\label{thm_duration}  
There exist families of curves $\set{\Gamma^s_\pm}_{s\in\R}$ such that
conditionally on a successful phase slip,
\begin{align}
\label{tba14a}
\lim_{\sigma\to0} \Law \bigl( \theta(\ph_{\tau_0}) -
\theta(\ph_{\tau^s_-}) - \abs{\log\sigma} \bigr)
 &= \Law \biggl( \frac{Z}{2} - \frac{\log(2)}{2} + s 
\biggr)\;, \\
\label{tba14b} 
\lim_{\sigma\to0} \Law \bigl( \theta(\ph_{\tau^s_+}) -
\theta(\ph_{\tau_0}) - \abs{\log\sigma} \bigr)
 &= \Law \Bigl( \Theta + s 
\Bigr)\;, \\
\label{tba14c} 
\lim_{\sigma\to0} \Law \bigl( \theta(\ph_{\tau^s_+}) -
\theta(\ph_{\tau^s_-}) - 2\abs{\log\sigma} \bigr)
 &= \Law \Bigl( Z + 2s
\Bigr)\;,
\end{align}
where $Z$ denotes a standard Gumbel variable, $\Theta=-\log|N|$, and $N$ is a
standard normal random variable. The curves $\Gamma^s_\pm$ are ordered in the
sense that if $s_1<s_2$, then $\Gamma^{s_1}_\pm$ lies below $\Gamma^{s_2}_\pm$. 
Furthermore, $\Gamma^s_+$ converges to the unstable orbit $\set{r=0}$ as
$s\to-\infty$, and to the stable orbit $\set{r=1/2}$ as
$s\to\infty$. Similarly, $\Gamma^s_-$ converges to the unstable orbit
$\set{r=0}$ as $s\to\infty$, and to the stable orbit $\set{r=-1/2}$ as
$s\to-\infty$. 
\end{theorem}

We give the proof in Appendix~\ref{ssec_proof_duration}, along with more details
on how to construct the curves $\Gamma^s_\pm$. In a nutshell, they are obtained
by letting evolve under the deterministic flow the curves
$\set{r=\pm\delta\sqrt{2\lambda_+T_+\hper(\ph)}}$ introduced in the previous
section. The parameter $s$ plays an analogous role as $T(x_0,b)$
in~\eqref{lrp05}.


\section{Conclusion and outlook}
\label{sec_conclusion}

Let us restate our main results in an informal way.
Theorem~\ref{thm_convergence_Gumbel} shows that in the weak-noise limit, the
position of the center $\ph_{\tau_0}$ of a phase slip, defined by the crossing
location of the unstable orbit, behaves like 
\begin{equation}
 \label{conc1}
 \theta(\ph_{\tau_0}) \simeq |\log\sigma| + \lambda_+T_+ Y^\sigma 
 + \frac{Z}{2} - \frac{\log2}{2}\;,
\end{equation} 
where $Y^\sigma$ is an asymptotically geometric random variable with success
probability of order $\e^{-I_\infty/\sigma^2}$, and $Z$ is a standard Gumbel
random variable. This expression is dominated by the term $Y^\sigma$, which
accounts for exponentially long waiting times between phase slips. The term
$\abs{\log\sigma}$ is responsible for the cycling phenomenon, and the term
$(Z-\log(2))/2$ determines the shape of the cycling profile. 

Theorem~\ref{thm_duration} shows in particular that the duration of a phase
slip behaves like 
\begin{equation}
 \label{conc2}
 \theta(\ph_{\tau_+}) - \theta(\ph_{\tau_-}) 
 \simeq 2|\log\sigma| + Z + 2s\;,
\end{equation} 
where $s$ is essentially the deterministic time required to travel between
$\sigma$-neighbourhoods of the orbits, while the other two terms account for
the time spent near the unstable orbit. The dominant term here is
$2\abs{\log\sigma}$, which reflects the intuitive picture that noise enlarges
the orbit to a thickness of order $\sigma$, outside which the deterministic
dynamics dominates. The phase slip duration is split
into two contributions from before and after crossing the unstable orbit, of
respective size $(Z-\log2)/2 + s$ and $\Theta + s$.

Decreasing the noise intensity has two main effects. The first one is to
increase the duration of phase slips by an amount  $2|\log\sigma|$, which is
due to the longer time spent near the unstable orbit. The second effect is to
the shift of the phase slip location by an amount $|\log\sigma|$, which results
in log-periodic oscillations. Note that other quantities of interest can be
deduced from the above expressions, such as the distribution of residence times,
which are the time spans separating phase slips when the system is in a
stationary state. The residence-time distribution is given by the sum of an
asymptotically geometric random variable and a so-called logistic random
variable, i.e., a random variable having the law of the difference of two
independent Gumbel variables, with density
proportional to $1/\cosh^2 (\theta)$~\cite{BG9}. 

The connection between first-exit distributions and extreme-value theory is
partially understood in the context as residual lifetimes, as summarized in
Section~\ref{ssec_Bakhtin}. It is probable that other connections remain to be
discovered. For instance, functional equations satisfied by the Gumbel
distribution seem to play an important r\^ole. One of them is the equation 
\begin{equation}
\label{conc3}
\Lambda\bigl(x-\log 2\bigr)^2 = \Lambda(x)
\end{equation}
which results from the Gumbel law being max-stable. Another one is the equation
\begin{equation}
\label{conc4}
\Lambda\bigl(\e^{-x}\bigr) = \e^{-\Lambda(x)}
\end{equation}
which appears in the context of the residual-lifetime interpretation. These
functional equations may prove useful to establish other connections with
critical phenomena and discrete scale invariance. 


\appendix

\section{Proof of Theorem~\ref{thm_convergence_Gumbel}}
\label{sec_proof_Gumbel}


\subsection{Dynamics near the unstable orbit}
\label{ssec_pG_u}

To prove convergence in law of certain random variables, we will work with
characteristic functions. The following lemma allows to compare characteristic
functions of random variables that are only known in a coarse-grained sense,
via probabilities to belong to small intervals of size $\Delta$. 

\begin{lemma}
\label{lem_pGu} 
Let $X,X_0$ be real-valued random variables. Assume there exist constants 
$a < b \in\R$, $\alpha, \beta>0$ such that as $\Delta\to0$, 
\begin{enum}
\item 	$\prob{X_0\not\in[a,b]} = \Order{\Delta^\alpha}$,
\item 	for any $k\in\Z$ such that $I_k=[k\Delta,(k+1)\Delta]$ intersects
$[a,b]$, 
\begin{equation}
 \label{pG_u01}
 \prob{X\in I_k} = \prob{X_0\in I_k} \bigl[ 1+\Order{\Delta^\beta} \bigr]\;.
\end{equation} 
\end{enum}
Then 
\begin{equation}
 \label{pG_u02}
 \bigl| \bigexpec{\e^{\icx\eta X}} - \bigexpec{\e^{\icx\eta X_0}}\bigr|
 \leqs 4 \sin \biggl( \frac{|\eta|\Delta}{2} \biggr) +
\Order{\Delta^{\alpha\wedge\beta}}
\end{equation} 
holds for all $\eta\in\R$. 
\end{lemma}
\begin{proof}
We start by noting that~\eqref{pG_u01} implies 
\begin{equation}
 \label{pG_u03:1}
 \bigprob{X\in[a,b]} 
 = \bigprob{X_0\in[a,b]} \bigl[ 1+\Order{\Delta^\beta} \bigr]
 = \bigl[ 1-\Order{\Delta^\alpha} \bigr] \bigl[ 1+\Order{\Delta^\beta} \bigr]\;,
\end{equation} 
and thus $\prob{X\not\in[a,b]} = \Order{\Delta^{\alpha\wedge\beta}}$. 
It follows that 
\begin{align}
\nonumber
 \bigl| \bigexpec{\e^{\icx\eta X}1_{\set{X\not\in[a,b]}}} 
 - \bigexpec{\e^{\icx\eta X_0}1_{\set{X_0\not\in[a,b]}}}\bigr|
 &\leqs \bigprob{X\not\in[a,b]} + \bigprob{X_0\not\in[a,b]} \\
 &= \Order{\Delta^{\alpha\wedge\beta}}\;.
 \label{pG_u03:2}
\end{align} 
Next, using the triangular inequality, we obtain for all $k$ such that 
$I_k\cap[a,b]\neq\emptyset$ 
\begin{equation}
 \label{pG_u03:3}
 \bigl| \bigexpec{\e^{\icx\eta X}1_{\set{X\in I_k}}} 
 - \bigexpec{\e^{\icx\eta X_0}1_{\set{X_0\in I_k}}}\bigr|
 \leqs A_k + B_k + C_k\;,
\end{equation} 
where 
\begin{align}
\nonumber
A_k &= \bigl| \bigexpec{\e^{\icx\eta X}1_{\set{X\in I_k}}} 
 - \e^{\icx\eta k\Delta} \bigprob{X\in I_k}\bigr|\;, \\
\nonumber
B_k &= \bigl| \e^{\icx\eta k\Delta} 
 \bigl(\bigprob{X\in I_k}-\bigprob{X_0\in I_k} \bigr) \bigr|\;, \\
C_k &= \bigl| \e^{\icx\eta k\Delta} \bigprob{X_0\in I_k} 
 - \bigexpec{\e^{\icx\eta k X_0}1_{\set{X_0\in I_k}}} \bigr|\;.
 \label{pG_u03:4}
\end{align}
Using the fact that $|\e^{\icx\eta(x-k\Delta)}-1| \leqs 2\sin(|\eta|\Delta/2)$
for all $x\in I_k$, we obtain 
\begin{align}
\nonumber
 A_k 
 &\leqs \int_{I_k} \bigl| \e^{\icx\eta x} - \e^{\icx\eta k\Delta}\bigr|
 \bigprob{X\in\6x}
 \leqs 2 \sin \biggl( \frac{|\eta|\Delta}{2} \biggr) 
 \bigprob{X\in I_k}\;, \\
 C_k 
 &\leqs \int_{I_k} \bigl| \e^{\icx\eta k\Delta} - \e^{\icx\eta x} \bigr|
 \bigprob{X_0\in\6x}
 \leqs 2 \sin \biggl( \frac{|\eta|\Delta}{2} \biggr) 
 \bigprob{X_0\in I_k}\;.
 \label{pG_u03:5}
\end{align} 
In addition, \eqref{pG_u01} implies that 
$B_k \leqs \Order{\Delta^\beta} \prob{X_0\in I_k}$. 
Hence the result follows by summing~\eqref{pG_u03:3} over all
$k$ and adding~\eqref{pG_u03:2}. 
\end{proof}

Consider now the solution $(r_t,\ph_t)$ of~\eqref{exit03}, starting in
$(0,-\delta)$. Recall that $\tau_0$ denotes the first-hitting time of the
unstable orbit at $r=0$. In addition, we denote by $\tau_{-2\delta}$ the
first-hitting time of the line $\set{r=-2\delta}$. 

\begin{prop}
\label{prop_pG_u1} 
For sufficiently small $\delta>0$, 
\begin{equation}
 \label{pG_u04}
 \lim_{\sigma\to0} 
 \bigecondin{0,-\delta}
 {\e^{\icx\eta(\theta_\delta(\ph_{\tau_0})-\abs{\log\sigma})}}
 {\tau_0 < \tau_{-2\delta}}
 = 2^{-\icx\eta/2} \Gamma \biggl( 1 - \icx \frac{\eta}{2}\biggr) +
\Order{\delta}\;,
\end{equation} 
where $\Gamma$ is Euler's Gamma function. 
\end{prop}
\begin{proof}
We will use several relations from~\cite[Sections~6 and 7]{BG_periodic2}, which
apply to the process killed whenever $r_t$ leaves the interval $(-2\delta,0)$.
For $\ell\in\N$ and $s\in[0,1)$, let 
\begin{align}
\nonumber
Q_{\widetilde\Delta}(0,\ell+s) &= 
\bigprobin{0,-\delta}{\ph_{\tau_0} \in 
[\ell+s,\ell+s+\widetilde\Delta]} \\
&= \bigprobin{0,-\delta}{\theta_\delta(\ph_{\tau_0}) \in 
[t,t+\Delta]}\;,
\label{pG_u05:1} 
\end{align}
where $t=\theta_\delta(\ell+s)=\ell\lambda_+T_++\theta_\delta(s)$ and 
$\Delta=\theta'_\delta(s)\widetilde\Delta+\Order{\widetilde\Delta^2}$. 

By~\cite[(7.14) and (7.18)]{BG_periodic2}, we also have 
\begin{align}
\nonumber
 Q_{\widetilde\Delta}(0,\ell+s) 
 &= C(\sigma)\widetilde\Delta \theta'_\delta(s)
A\bigl(t-|\log\sigma|+\Order{\delta\ell}\bigr)
\bigl[1+\Order{\widetilde\Delta^\beta}\bigr] \\
&= C(\sigma)\Delta
A\bigl(t-|\log\sigma|+\Order{\delta\ell}\bigr)
\bigl[1+\Order{\Delta^\beta}\bigr] 
 \label{pG_u05:2}
\end{align} 
for some constants $C(\sigma),0<\beta<1$, where
$A(t)=\e^{-2t-\frac12\e^{-2t}}$. It follows that 
\begin{equation}
 \label{pG_u05:21}
 \bigprobin{0,-\delta}{\theta_\delta(\ph_{\tau_0}) - |\log\sigma| \in 
[t,t+\Delta]}
= C(\sigma)\Delta
A\bigl(t+\Order{\delta\ell}\bigr)
\bigl[1+\Order{\Delta^\beta}\bigr]\;.
\end{equation} 
We now apply Lemma~\ref{lem_pGu}, where $X_0$ is a random variable with density
proportional to $A(t-|\log\sigma|+\Order{\delta\ell})
[1+\Order{\Delta^\beta}]$, and $X$ is the random variable
$\theta_\delta(\ph_{\tau_0})-|\log\sigma|$, conditional on
$\tau_0<\tau_{-2\delta}$,
i.e.\ $\prob{X\in B} =
\pcond{\theta_\delta(\ph_{\tau_0})-|\log\sigma|\in B}{\tau_0<\tau_{-2\delta}}$. 
We choose cut-offs $a=-\frac12\log(\log\Delta^{-1})$ and $b=\log(\Delta^{-1})$. 
This guarantees, on the one hand, that $\prob{X_0\not\in[a,b]} =
\Order{\Delta^{1/2}}$. On the other hand, $|A'(x)/A(x)|$ is bounded above by
$\Order{\log\Delta^{-1}}$ on $[a,b]$, which allows to check that 
Condition~\eqref{pG_u01} holds. It follows that 
\begin{align}
\nonumber
 \bigecondin{0,-\delta}
 {\e^{\icx\eta(\theta_\delta(\ph_{\tau_0})-\abs{\log\sigma})}}
 {\tau_0 < \tau_{-2\delta}} 
 ={}& 
 \widetilde C(\sigma) \int_{\theta_\delta(0)-\abs{\log\sigma}}^\infty 
 \e^{\icx\eta t} A(t) \6t \bigl[1+\Order{\delta}+\Order{\Delta^\beta}\bigr] \\
 &{} + \Order{\Delta^{1/2\wedge\beta}} + 
 \cO \biggl( \frac{|\eta|\Delta}{2} \biggr)\;,
\label{pG_u05:3}
\end{align} 
where $\widetilde C(\sigma)=C(\sigma)/\prob{\tau_0<\tau_{-2\delta}}$. 
The change of variables $v=\frac12\e^{-2t}$ yields 
\begin{align}
\nonumber
 \int_{\theta_\delta(0)-\abs{\log\sigma}}^\infty \e^{\icx\eta t} A(t) \6t
 &= 2^{-\icx\eta/2} \int_0^{\e^{-2\theta_\delta(0)}/2\sigma^2}
 v^{-\icx\eta/2} \e^{-v} \6v \\
 &= 2^{-\icx\eta/2} \Gamma \biggl( 1 - \frac{\icx\eta}{2} \biggr) 
 + \Order{\e^{-\Order{\delta^2/\sigma^2}}}\;.
 \label{pG_u05:4}
\end{align} 
Plugging this into~\eqref{pG_u05:3}, taking first the limit $\sigma\to0$ and
then the limit $\Delta\to0$ shows that $\widetilde C(0)=1+\Order{\delta}$ (by
evaluating in $\eta=0$) and proves~\eqref{pG_u04}.
\end{proof}

Note that the limit as $\delta\to0$ of the right-hand side of~\eqref{pG_u04}
is the characteristic function of $(Z-\log2)/2$, where $Z$ is a standard Gumbel
random variable. 


\subsection{Large deviations and dynamics far from the unstable orbit}
\label{ssec_pG_ldp}

We consider in this section the dynamics up to the time 
\begin{equation}
 \label{pG_ldp01}
 \tau_{-\delta} = \inf\setsuch{t>0}{r_t = -\delta}
\end{equation} 
the process enters a $\delta$-neighbourhood of the unstable periodic orbit. 
We will choose a sequence $(\delta_m)_{m\in\N}$, converging to zero as
$m\to\infty$, such that the optimal path $\gamma_\infty$ crosses the level
$-\delta_m$ when $\ph=m$. It follows from the behaviour of the
Hamiltonian~\eqref{ldp05} near the unstable orbit
(cf.~\cite[Section~4]{BG_periodic2}) that 
\begin{equation}
 \label{pG_ldp02}
 |\log\delta_m| = m\lambda_+T_+ + \Order{\e^{-m\lambda_+T_+}}
\end{equation} 
as $m\to\infty$, which shows that 
$\delta_m \simeq \e^{-m\lambda_+T_+}$ for large $m$. 

We wish to compute the rate function $I_m(\ph)$, corresponding to sample paths
reaching $r=-\delta_m$ near a given value of $\ph$. 

\begin{figure}
\centerline{\includegraphics*[clip=true,height=75mm]{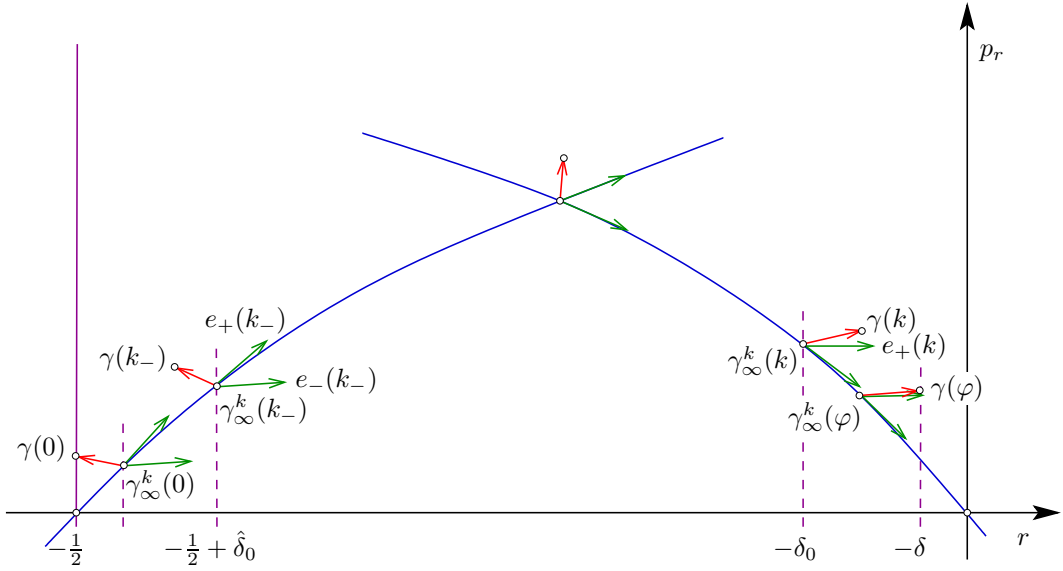}}
 \figtext{ 
	\writefig	0.9	0.55	$-\frac12$
	\writefig	2.45	0.55	$-\frac12+\hat\delta_0$
	\writefig	10.55	0.55	$-\delta_0$
	\writefig	12.15	0.55	$-\delta$
	\writefig	13.8	0.7	$r$
	\writefig	13.3	7.2	$p_r$
	\writefig	1.95	1.4	$\gamma^k_\infty(0)$
	\writefig	3.2	2.4	$\gamma^k_\infty(k_-)$
	\writefig	9.95	3.05	$\gamma^k_\infty(k)$
	\writefig	10.75	2.3	$\gamma^k_\infty(\ph)$
	\writefig	4.2	2.8	$e_-(k_-)$
	\writefig	3.0	3.6	$e_+(k_-)$
	\writefig	12.0	3.25	$e_+(k)$
	\writefig	0.5	1.9	$\gamma(0)$
	\writefig	1.6	3.1	$\gamma(k_-)$
	\writefig	11.8	3.6	$\gamma(k)$
	\writefig	12.65	2.7	$\gamma(\ph)$
 }
\caption[]{Construction of the optimal path $\set{\gamma(s)}_{0\leqs
s\leqs\ph}$ for a first exit at $\ph_{\tau_0}=\ph$. The path is constructed as a
perturbation of the translate $\gamma^k_\infty(s)=\gamma_\infty(s_k)$ of the
path minimizing the rate function in arbitrary time, which lies at the
intersection of the unstable and stable manifolds of the periodic orbits.  
}
\label{fig_ldp_section}
\end{figure}

\begin{prop}
\label{prop_pG_ldp1}
There exists a periodic function $P(\ph)$, reaching its maximum if and only
if $\ph\in\Z$, such that 
\begin{equation}
 \label{pG_ldp03}
 I_m \biggl( \hat\ph + \frac{|\log\delta_m|}{\lambda_+T_+} \biggr)
 = I_\infty - P(\hat\ph) \delta_m^2 
 + \Order{\delta_m^4 + \e^{-2\lambda_+T_+\hat\ph}}\;.
\end{equation} 
\end{prop}
\begin{proof}
The proof is based on similar considerations as
in~\cite[Section~4]{BG_periodic2}. We will construct a path $\gamma$ minimising
the rate function by perturbation of a translate of the optimal path
$\gamma_\infty$. This is justified by our assumption that $\gamma_\infty$ is a
unique minimizer for transitions in arbitrary time, up to translations.

We fix a small constant $\delta_0$, such that $\gamma_\infty$ reaches level
$-\delta_0$ for an integer value of $\ph$. Without loss of generality, we may
assume that $\gamma_\infty(0)=-\delta_0$. Then we can also find an integer
$\ell$ and a $\hat\delta_0 \leqs \delta_0$ such that
$\gamma_\infty(-\ell)=-1/2+\hat\delta_0$. The translate $\gamma_\infty^k =
\gamma_\infty(\cdot-k)$ crosses level $-\delta_0$ at $\ph=k$ and level
$-1/2+\hat\delta_0$ at $\ph=k_-:=k-\ell$. 

The Hamiltonian flow of~\eqref{ldp05} can be viewed as a time-dependent flow for
$(r,p_r)$ in which $\ph$ plays the r\^ole of
time~\cite[Section~2.2]{BG_periodic2}. We denote by $\pm\mu(s)$ the eigenvalues
of the linearised flow at $\gamma_\infty^k(s)$. Let 
\begin{equation}
 \label{pG_ldp04:1}
 \alpha(s,s_0) = \int_{s_0}^s \mu(u)\6u\;.
\end{equation} 
The principal solution $U(s,s_0)$ of the linearised flow has eigenvalues
$\e^{\pm\alpha(s,s_0)}$. We denote by $e_\pm(s)$ the associated eigenvectors, 
which satisfy $U(s,s_0) e_\pm(s_0)=\e^{\pm\alpha(s,s_0)}e_\pm(s)$. 
Consider a perturbed path given by 
\begin{equation}
 \label{pG_ldp04:2}
 \gamma(s) = \gamma_\infty^k(s) + a_se_+(s) + b_se_-(s)
\end{equation} 
(\figref{fig_ldp_section}). Then we have 
\begin{align}
\nonumber
a_s &= a_k\e^{\alpha(s,k)} + \Order{\norm{a^2+b^2}_\infty}\;, \\
b_s &= b_k\e^{-\alpha(s,k)} + \Order{\norm{a^2+b^2}_\infty}\;,
\label{pG_ldp04:3} 
\end{align}
where $\norm{\cdot}_\infty$ denotes the supremum over $[0,\ph]$. 

Given an interval $[u,s]$, let $I^0(u,s)$ denote the contribution of $[u,s]$ to
the integral defining the rate function of $\gamma_\infty^k$
(cf.~\eqref{ldp01}). Let $I^1(u,s)$ denote its analogue for the rate function of
$\gamma$. Then a computation similar to the one
in~\cite[Proposition~4.1]{BG_periodic2} shows that, up to a multiplicative error
$1+\Order{\delta_0}$, 
\begin{equation}
 \label{pG_ldp04:4}
 I^1(k,\ph) - I^0(k,\infty) = -\frac{\delta_0^2}{2}
\frac{\hper(\ph)}{\hper(0)^2}x_+^2 - \delta_0 b_k +
\Order{b_k^2,b_kx_+^2}\;,
\end{equation} 
where $x_+ = \e^{-\alpha(\ph,k)}$. In addition, the boundary condition
$\gamma(\ph)\in\set{r=-\delta_m}$ yields 
\begin{equation}
 \label{pG_ldp04:5}
 a_k = \delta_0 \frac{\hper(\ph)}{\hper(0)}x_+^2 - \delta_m x_+ 
 + \Order{b_kx_+^2}\;.
\end{equation} 
A similar analysis can be made in the vicinity of the stable periodic orbit,
and yields
\begin{equation}
 \label{pG_ldp04:6}
 I^1(0,k_-) - I^0(-\infty,k_-) = -\frac{\hat\delta_0^2}{2}
\frac{1}{\hper_-(0)}x_-^2 + \hat\delta_0 a_k\e^{-\alpha(k,k_-)} +
\Order{a_k^2,a_kx_-^2}\;,
\end{equation} 
where $x_-=\e^{-\alpha(k_-,0)}$, and $\hper_-$ is a periodic function related
to the linearisation at the stable orbit. Note that $\alpha(k,k_-)=:\alpha_0$
does not depend on $k$, but only on the time $\ell$ it takes for the optimal
path $\gamma_\infty$ to go from $-1/2+\hat\delta_0$ to $-\delta_0$. The boundary
condition $\gamma(0)\in\set{r=-1/2}$ yields the condition 
\begin{equation}
 \label{pG_ldp04:7}
 \e^{\alpha_0}b_k = -\hat\delta_0 x_-^2  + \Order{a_kx_-^2}\;.
\end{equation} 
Finally, the transition between times $k_-$ and $k$ yields 
\begin{equation}
 \label{pG_ldp04:8}
 I^1(k,k_-) - I^0(k,k_-) 
 = a_k c_+ + b_k c_- + \Order{a_k^2+b_k^2}\;,
\end{equation} 
where the constants $c_\pm$ depend only on the behaviour of
$\gamma_\infty^k$ on $[k_-,k]$, and are thus independent of $k$. Adding the
estimates~\eqref{pG_ldp04:4}, \eqref{pG_ldp04:6} and~\eqref{pG_ldp04:8}, and
using the boundary conditions, we obtain
\begin{equation}
 \label{pG_ldp04:9}
 I^1(\ph,0) - I_\infty = A\hper_+(\ph)x_+^2 + Bx_-^2 - C\delta_m x_+
 + \Order{x_+^4+x_-^4+\delta_m^4}
\end{equation} 
for some positive constants $A, B, C$. The optimal rate function is obtained by
minimizing~\eqref{pG_ldp04:9} under the constraint
$x_-x_+\e^{-\alpha_0}=\e^{-\alpha(\ph,0)}$. The result is that the minimum is
reached when 
\begin{equation}
 \label{pG_ldp04:10}
 x_+ = \frac{C}{2A\hper(\ph)} \delta_m 
 \bigl[ 1 + \Order{\delta_m^2} + \Order{\delta_m^{-4}\e^{-\alpha(\ph,0)}}
\bigr]\;,
\end{equation}
and has value 
\begin{equation}
 \label{pG_ldp04:11}
 I_m(\ph) - I_\infty = -\frac{C^2}{4A\hper(\ph)}\delta_m^2 
 + \Order{\delta_m^4 + \delta_m^{-2}\e^{-2\alpha(\ph,0)}}\;.
\end{equation} 
Noting that for large $m$, $\alpha(\ph,0) - \lambda_+T_+\ph$ 
is bounded below, which implies
$\e^{-\alpha(\ph,0)}=\Order{\delta_m\e^{-\lambda_+T_+\hat\ph}}$, finishes the
proof. 
\end{proof}


\subsection{Final steps of the proof}
\label{ssec_pG_proof}

Consider the random variable 
\begin{equation}
 \label{pG_p00}
 \hat\ph_m = \ph_{\tau_{-\delta_m}} -
\frac{|\log\delta_m|}{\lambda_+T_+}\;.
\end{equation} 
Note that~\eqref{pG_ldp02} implies $\hat\ph_m-\ph_{\tau_{-\delta_m}} +
m\to0$ as $m\to\infty$, meaning that asymptotically, $\hat\ph_m$ is just an
integer shift of $\ph_{\tau_{-\delta_m}}$. We introduce the integer-valued
random variable 
\begin{equation}
 \label{pG_p01}
 Y_m^\sigma = \biggintpart{\hat\ph_m + \frac12}\;,
\end{equation} 
which has the property 
\begin{equation}
 \label{pG_p02}
 \bigprob{Y_m^\sigma = n}
 = \biggprob{\hat\ph_m \in [n-\frac12,n+\frac12)}\;,
\end{equation} 
that is, $Y_m^\sigma$ counts the number of periods until the process reaches the
unstable orbit. The same argument as the one yielding~\eqref{rpm07} shows that
the random variable $Y_m^\sigma$ is asymptotically geometric. The
large-deviation principle implies that the success probability is of order
$\e^{-I_m/\sigma^2}$, with $I_m=I_\infty-P(0)\delta_m^2+\Order{\delta_m^4}$.

\begin{prop}
\label{prop_pG_p1}  
We have 
\begin{equation}
 \label{pG_p03}
 \lim_{m\to\infty} \Bigl(
 \lim_{\sigma\to0} \bigprob{ \bigl|\hat\ph_m  - Y_m^\sigma \bigr| > \eta }
 \Bigr)
 = 0 \;,
\end{equation} 
for all $\eta>0$. Therefore, $\hat\ph_m-Y_m^\sigma$
converges in distribution to $\delta_0$, the Dirac mass at $0$. 
\end{prop}
\begin{proof}
Proposition~\ref{prop_pG_ldp1} implies that the rate function $I_m(\hat\ph)$ has
at most one minimum in $[n-1/2,n+1/2)$, satisfying
$\hat\ph=n+\Order{\delta_m^2}+\Order{\e^{-2n\lambda_+T_+}}$. 
We decompose 
\begin{equation}
 \label{pG_p04:1}
 \bigprob{ \bigl|\hat\ph_m  - Y_m^\sigma \bigr| > \eta }
 \leqs \bigprob{ \bigl|\hat\ph_m  - Y_m^\sigma \bigr| > \eta , Y_m^\sigma > 2m}
 + \bigprob{ Y_m^\sigma \leqs 2m}\;.
\end{equation} 
If $Y_m^\sigma > 2m$, then $\e^{-Y_m^\sigma\lambda_+T_+} =
\Order{\delta_m^2}$. Thus as $\sigma\to0$, the first term on the right-hand
side becomes bounded by $1_{\set{\eta<\Order{\delta_m^2}}}$, which converges to
$0$ as $m\to\infty$. By the large-deviation principle, the second term on the
right-hand side has order $2m\e^{-c/\sigma^2}$ for some $c>0$, which goes to
zero as
$\sigma\to0$. 
\end{proof}

By~\eqref{pG_p00} and the definition~\eqref{cycling06} of
$\theta_\delta$ we have
\begin{equation}
 \label{pG_p05}
 \theta(\ph_{\tau_0}) - |\log\sigma| - \lambda_+T_+Y_m^\sigma 
 = \bigl[\theta_{\delta_m}(\ph_{\tau_0}) - |\log\sigma| -
\lambda_+T_+\ph_{\tau_{-\delta_m}} \bigr]
 + \lambda_+T_+(\hat\ph_m - Y_m^\sigma)\;.
\end{equation} 
By L\'evy's continuity theorem, Proposition~\ref{prop_pG_u1} and periodicity,
conditionally on $\tau_0 < \tau_{-2\delta_m}$, the term in square brackets
converges in distribution to $(Z-\log2)/2$ as $\sigma\to0$ and $m\to\infty$ (in
that order). The second term on the right-hand side converges to $0$, as we have
just seen. This proves the result conditionally on $\tau_0 < \tau_{-2\delta_m}$.
The result remains true unconditionally because if 
$\hat\tau_{-2\delta_m} = \inf\setsuch{t>\tau_{-\delta_m}}{r_t=-2\delta_m}$,
then 
\begin{equation}
 \label{pG_p06}
 \bigpcond{\hat\tau_{-2\delta_m}<\tau_0}{\tau_0\in B} =
\Order{\e^{-c\delta_m^2/\sigma^2}}
\end{equation} 
for some constant $c>0$, as a consequence
of~\cite[Proposition~4.2]{BG_periodic2}. This reflects the fact that it is more
expensive, in terms of rate function, to move back and forth between levels
$-\delta_m$ and $-2\delta_m$ before reaching the unstable orbit, than to go
directly from level $-\delta_m$ to $0$ (see also the renewal equation 
in~\cite{BG7}). 
\qed


\section{Duration of phase slips}
\label{sec_proof_duration}


\subsection{Proof of Theorem~\ref{thm_exit}}
\label{ssec_proof_exit}


We will start by characterising the exit distribution from a small strip of
width of order $h_0=\sigma^\gamma$ around the unstable orbit. 
We set 
\begin{equation}
 \label{px01}
 \tilde\tau_{\pm h_0} = \inf\bigsetsuch{t>0}{r_t = \pm
h_0\sqrt{2\lambda_+T_+\hper(\ph_t)}}\;.
\end{equation} 
The following result is an adaptation of~\cite[Theorem~2.1]{Day2} to the
nonlinear, $\ph$-dependent situation. 

\begin{lemma}
\label{lem_exit1}
Fix an initial condition $(\ph_0,0)$ and a constant $h_0=\sigma^\gamma$ for
$\gamma\in(1/2,1)$. Then the solution of~\eqref{exit03} satisfies 
\begin{equation}
 \label{px02}
 \lim_{\sigma\to0} 
 \Law \biggl( \theta(\ph_{\tilde\tau_{h_0}}) - \theta(\ph_0) -
 \log\biggl( \frac{h_0}{\sigma} \biggr) 
  \biggm| \tilde\tau_{h_0} < \tilde\tau_{-h_0} \biggr)
 = \Law \biggl( \Theta + \frac12\log(2\lambda_+)  \biggr) \;, 
\end{equation} 
where $\Theta=-\log|N|$ and $N$ is a standard normal random variable. 
\end{lemma}
\begin{proof}
Let $\hat\tau = \tau_{h_0} \wedge \tau_{-h_0}$, and $\nu=\sign(r_{\hat\tau})$. 
We will introduce several events that have probabilities going to $1$ as
$\sigma\to0$. The first event is 
\begin{equation}
 \label{px03:1}
 \Omega_1 = \biggl\{ \biggl| \ph_t - \ph_0 - \frac{t}{T_+} \biggr| 
 \leqs M(h_0^2 t + h_0) 
 \quad\forall t \leqs \hat\tau \wedge \frac{1}{h_0} \biggr\}\;,
\end{equation} 
where $M$ is such that $|b_\ph(r,\ph)|\leqs Mr^2$ for all $(r,\ph)$.
Then~\cite[Proposition~6.3]{BG_periodic2} shows that 
\begin{equation}
 \label{px03:2}
 \fP(\Omega_1^c) \leqs \e^{-\kappa_1h_0/\sigma^2}
\end{equation} 
for a constant $\kappa_1>0$. On $\Omega_1$, the phase $\ph_t$ remains
$h_0$-close to $\ph_0+t/T_+$. Hence the equation for $r_t$ can be written as 
\begin{equation}
 \label{px03:3}
 \6r_t = \bigl[\lambda_+ r_t + b_r(r_t,\ph_t) \bigr] \6t 
 + \bigl[g_0(t)+g_1(r_t,\ph_t,t) \bigl] \6W_t
\end{equation} 
where $g_0(t)=g_r(0,\ph_0+t/T_+)$, and $g_1=\Order{|r|+h_0}$ on $\Omega_1$. 
The solution can be represented as 
\begin{equation}
 \label{px03:4}
 r_t = \e^{\lambda_+t} 
 \biggl[
 \sigma \int_0^t \e^{-\lambda_+s}g_0(s)\6W_s
 + \sigma \int_0^t \e^{-\lambda_+s}g_1(r_s,\ph_s,s)\6W_s
 + \int_0^t \e^{-\lambda_+s} b_r(r_s,\ph_s)\6s
 \biggr]\;.
\end{equation} 
Let $Y_t$ denote the second integral in~\eqref{px03:4}, and define 
\begin{equation}
 \label{px03:5}
 \Omega_2(t) = 
 \biggl\{
  \sup_{0\leqs s\leqs \hat\tau\wedge t} |Y_s| > H
 \biggr\}\;.
\end{equation}
Then the Bernstein-type estimate~\cite[Theorem~37.8]{RogersWilliams}
yields 
\begin{equation}
 \label{px03:6}
 \fP\bigl( \Omega_1\cap\Omega_2(t)^c \bigr) \leqs \e^{-\kappa_2 H^2/h_0^2}
\end{equation} 
for a constant $\kappa_2>0$, uniformly in $t$.
Evaluating~\eqref{px03:4} in $\hat\tau$, we obtain that on
$\Omega_1\cap\Omega_2(t)$, 
\begin{equation}
 \label{px03:7}
 \nu h_0\sqrt{2\lambda_+T_+\hper(\ph_{\hat\tau})} = \e^{\lambda_+\hat\tau}
 \biggl[
 \sigma \int_0^{\hat\tau} \e^{-\lambda_+ s} g_0(s) \6W_s
 + \Order{\sigma H + h_0^2}
 \biggr]\;.
\end{equation} 
Now we decompose 
\begin{equation}
 \label{px03:8}
 \int_0^{\hat\tau} \e^{-\lambda_+ s} g_0(s) \6W_s
 = \int_0^{\infty} \e^{-\lambda_+ s} g_0(s) \6W_s 
 - \int_{\hat\tau}^\infty \e^{-\lambda_+ s} g_0(s) \6W_s
 = \sqrt{v_\infty} N - R\;,
\end{equation} 
where $N$ is a standard normal random variable, and 
(cf.~\cite[(2.29)]{BG_periodic2})
\begin{equation}
 \label{px03:8B}
 v_\infty = \int_0^\infty \e^{-2\lambda_+ s} D_{rr}(\ph_0+s/T_+) \6s
 = T_+ \hper(\ph_0)\;.
\end{equation}
The remainder $R$ satisfies 
\begin{equation}
 \label{px03:9}
 \expec{R^2} = \frac{1}{2\lambda_+} \bigexpec{\e^{-2\lambda_+\hat\tau}}
 = \biggOrder{\frac{\sigma^2}{h_0^2}}\;,
\end{equation} 
since~\eqref{px03:7} implies that $\e^{-\lambda_+\hat\tau}h_0/\sigma$ is bounded
(see also~\cite{Kifer}). This shows that if we set $\Omega_3=\set{|R|>H}$, then 
by Markov's inequality there is a constant $\kappa_3>0$ such that 
\begin{equation}
 \label{px03:10}
 \fP(\Omega_3^c) \leqs \kappa_3 \frac{\sigma^2}{h_0^2H^2}\;.
\end{equation} 
Taking the logarithm of the absolute value of~\eqref{px03:7}, we find that on 
$\Omega_1\cap\Omega_2\cap\Omega_3$, 
\begin{equation}
 \label{px03:11}
 \log h_0 + \frac12 \log(2\lambda_+T_+\hper(\ph_{\hat\tau})) 
 = \lambda_+\hat\tau + \log\sigma + \frac12\log(T_+ \hper(\ph_0)) + \log|N|
 + \biggOrder{H+\frac{h_0^2}{\sigma}}\;.
\end{equation} 
Noting that on $\Omega_1$, $\hat\tau = T_+(\ph_{\hat\tau}-\ph_0) + \Order{h_0}$,
and recalling
the definition~\eqref{cycling05} of $\theta$, we obtain 
\begin{equation}
 \label{px03:12}
 \lim_{\sigma\to0} 
 \Law \biggl( \theta(\ph_{\hat\tau}) - \theta(\ph_0) -
 \log\biggl( \frac{h_0}{\sigma} \biggr) \biggr)
 = \Law \biggl( \Theta + \frac12\log(2\lambda_+)  \biggr) \;, 
\end{equation} 
by choosing $H=\sigma^{(1-\gamma)/2}$ and
taking the limit $\sigma\to0$. Finally, \eqref{px03:7} implies that
$\nu=\sign(r_{\hat\tau})$ converges to $\sign(N)$ as $\sigma\to0$, so that in
this limit $\prob{\nu=1}=\prob{\nu=-1}=1/2$. The result follows. 
\end{proof}

The following result shows that after time $\tau_{h_0}$, the system essentially
follows the deterministic dynamics. 

\begin{lemma}
\label{lem_exit2}
Let $h_0$ be as in the previous lemma. 
Fix an initial condition $(r_1,\ph_1)$ such that $r_1=h_0\sqrt{\hper(\ph_1)}$,
and a curve $\set{r=\delta\rho(\ph)}$, where $\rho$ is a continuous function
such that $0<\rho_0\leqs\rho(\ph)\leqs1$ for all $\ph$. 
Denote by 
\begin{itemiz}
\item $\tau_\delta$ the first-passage time at $r=\delta\rho(\ph)$ of the
solution $(r_t,\ph_t)$ starting in $(r_1,\ph_1)$, 
\item $\tau^{\det}_\delta$ the 
first-passage time at $r=\delta\rho(\ph)$ of the deterministic solution
$(r^{\det}_t,\ph^{\det}_t)$ starting in $(r_1,\ph_1)$.
\end{itemiz}
Then 
\begin{equation}
 \label{px04}
 \lim_{\sigma\to0} \ph_{\tau_\delta} = \ph^{\det}_{\tau^{\det}_\delta}
\end{equation} 
for sufficiently small $\delta$, where the convergence is in probability. 
\end{lemma}
\begin{proof}
The difference $\zeta_t=(r_t-r^{\det}_t,\ph_t-\ph^{\det}_t)$ satisfies an
equation of the form 
\begin{equation}
 \label{px05:1}
 \6\zeta_t = A(t)\zeta_t\6t + \sigma g(\zeta_t,t)\6W_t 
 + b(\zeta_t,t)\6t\;,
\end{equation} 
where $A(t)$ is a matrix with top left entry $a(t)=\lambda_+ +
\Order{r^{\det}_t}$, and the other entries of order $r^{\det}_t$ or
$(r^{\det}_t)^2$, while $b$ has order $\norm{\zeta}^2$. Then 
\begin{equation}
 \label{px05:2}
 \zeta_t = \zeta^0_t + \zeta^1_t 
 = \sigma\int_0^t U(t,s) g(\zeta_s,s)\6W_s 
 + \int_0^t U(t,s) b(\zeta_s,s)\6s\;,
\end{equation} 
where $U(t,s)$ is the fundamental solution of $\dot\zeta = A(t)\zeta$. 
Denote by $\alpha(t,s)$ the integral of $a(u)$ between $s$
and $t$. One can show that there exist constants $M,c>0$ such that 
\begin{equation}
 \label{px05:3}
 \norm{U(t,s)} \leqs M \e^{\alpha(t,s)+c\delta(t-s)}
 \qquad
 \forall t>s>0\;.
\end{equation} 
We want to show that with probability going to $1$ as $\sigma\to0$, $\zeta_t$
remains of order $h_1\e^{\alpha(t,0)}$ up to time
$\tau_\delta\vee\tau^{\det}_\delta$, where $h_1\to0$ as $\sigma\to0$. Set 
\begin{equation}
 \label{px05:4}
 \tau_1 = \inf\bigsetsuch{t}{
 \norm{\zeta_t} > h_1\e^{\alpha(t,0)}}\;.
\end{equation} 
A Bernstein-type estimate shows that 
\begin{equation}
 \label{px05:5}
 \biggprob{\sup_{0\leqs s\leqs t} \e^{-\alpha(s,0)} \norm{\zeta^0_s} > h_1}
 \leqs 
 \exp \biggset{-\kappa\frac{h_1^2}{\sigma^2} \e^{-2\alpha(t,0)-2ct}}\;.
\end{equation} 
Furthermore, there exists a constant $M>0$ such that 
\begin{equation}
 \label{px05:6}
 \norm{\zeta^1_{t\wedge\tau_1}} \leqs M h_1^2 \e^{2\alpha(t,0)+2ct}\;.
\end{equation} 
Since $\e^{\alpha(\tau^{\det}_\delta,0)} = \Order{\delta/h_0}$, we can find, for
sufficiently small $\delta$, an $h_1$ such that 
\begin{equation}
 \label{px05:7}
 \sigma \e^{\alpha(t,0)+ct} \ll h_1 \ll \e^{-\alpha(t,0)-ct}
\end{equation} 
for all $t \leqs \tau_\delta\vee\tau^{\det}_\delta$. This shows that with
probability going to $1$ as $\sigma\to0$, $\norm{\zeta_t} \leqs
(h_1/2)\e^{\alpha(t,0)}$ up to
time $\tau_1\wedge(\tau_\delta\vee\tau^{\det}_\delta)$, 
and thus $\tau_1 > \tau_\delta\vee\tau^{\det}_\delta$. This proves the claim.
\end{proof}

To finish the proof of Theorem~\ref{thm_exit}, we consider the deterministic
dynamics of~\eqref{exit03}. Using $\ph$ as new time variable, the dynamics can
be written as $\6r/\6\ph = f_r(r,\ph)/f_\ph(r,\ph)$. If we set
$y=r/\sqrt{\hper(\ph)}$, it follows from~\eqref{cycling03}
and~\eqref{cycling05} that this ODE is equivalent to 
\begin{equation}
 \label{px06}
 \frac{\6y}{\6\theta} = y + \tilde b(y,\theta)\;,
\end{equation} 
where $\tilde b(y,\theta) = \Order{y^2}$. Standard perturbation
theory shows that when starting with a small initial condition $y_0=h_0$,
solutions of this equation are of the form 
\begin{equation}
 \label{px07}
 y(\theta) = h_0 \e^{\theta-\theta_0} \bigl[ 1 + \Order{h_0
\e^{\theta-\theta_0}} \bigr]\;.
\end{equation} 
Thus $y(\theta)$ reaches $\delta$ when 
$\theta-\theta_0 = \log(\delta/h_0) + \Order{\delta}$. Combining this
with~\eqref{px02} yields the result.
\qed


\subsection{Proof of Theorem~\ref{thm_duration}}
\label{ssec_proof_duration}

Fix a small value of $\delta$, and consider an orbit starting at time
$\tilde\tau_\delta$ on the curve $\set{y=\delta}$. For $u\geqs0$, let 
$\Gamma^u_+(\delta)$ be the image after time $u$ of $\set{y=\delta}$ under the
deterministic flow~\eqref{px06}. Theorem~\ref{thm_exit} and
Lemma~\ref{lem_exit2} imply that the sample path will hit $\Gamma^u_+(\delta)$ 
at a time $\ph_{\tau_+(\delta)}$ such that 
\begin{equation}
 \label{pd01}
 \theta(\ph_{\tau_+(\delta)}) -\theta(\ph_{\tau_0}) - |\log\sigma| 
 \to 
 \Theta + \frac{\log(2\lambda_+\delta^2)}{2} + u + \Order{\delta}
\end{equation} 
as $\sigma\to0$. Setting $u = s - \log(\delta)$ and taking the limit
$\delta\to0$, we obtain~\eqref{tba14b}, where 
\begin{equation}
 \label{pd02}
 \Gamma^s_+ = \lim_{\delta\to0} \Gamma^{s-\log\delta}_+(\delta)\;.
\end{equation} 
Since $y_u \simeq y_0\e^u = \delta\e^{u}$ for small $\delta\e^{u}$, this limit
is indeed well-defined. This is the equivalent of the regularization procedure
used in~\eqref{lrp05}. 

Now~\eqref{tba14a} is obtained in an analogous way, using
Proposition~\ref{prop_pG_u1}, and~\eqref{tba14c} follows from the two previous
results, using Lemma~\ref{lem_Bakhtin}. 
\qed

We note that more explicit forms of the curves $\Gamma^s_\pm$ can be provided
if we can find a change of variables $y\mapsto\hat y$ such that 
\begin{equation}
 \label{pd03}
 \frac{\6\hat y}{\6\theta} = F(\hat y) = \hat y + \hat b(\hat y)\;,
\end{equation} 
in which the right-hand side does not depend on $\theta$. Such a construction
can be performed, for instance, in the averaging regime. In these variables,
the curves $\Gamma^s_\pm$ are simply horizontal lines of constant $\hat y$, and
their dependence on $s$ can be determined by solving a one-dimensional ODE,
yielding expressions similar to~\eqref{lrp05}.

\small
\bibliography{../IHP}
\bibliographystyle{abbrv}               

\newpage
\tableofcontents

\goodbreak

\vfill

\bigskip\bigskip\noindent
{\small 
Nils Berglund \\ 
Universit\'e d'Orl\'eans, Laboratoire {\sc Mapmo} \\
{\sc CNRS, UMR 7349} \\
F\'ed\'eration Denis Poisson, FR 2964 \\
B\^atiment de Math\'ematiques, B.P. 6759\\
45067~Orl\'eans Cedex 2, France \\
{\it E-mail address: }{\tt nils.berglund@univ-orleans.fr}
}


\end{document}